\documentclass{amsart}%
\usepackage{amssymb}
\usepackage{amsfonts}
\usepackage{amsmath}
\usepackage{graphicx}%
\newtheorem{theorem}{Theorem}
\theoremstyle{plain}

\newtheorem{conjecture}{Conjecture}
\newtheorem{corollary}{Corollary}

\newtheorem{definition}{Definition}
\newtheorem{example}{Example}

\newtheorem{lemma}{Lemma}

\newtheorem{proposition}{Proposition}

\numberwithin{equation}{section}
\begin{document}
\title[$EIP$ on $S(n,m)$]{The Edge-Isoperimetric Problem on Sierpinski Graphs}
\author{L. H. Harper}
\address{Department of Mathematics\\
University of California-Riverside}
\email{harper@math.ucr.edu}
\date{December 9, 20015}

\begin{abstract}
Some families of graphs, such as the $n$-cubes and Sierpinski gaskets are
self-similar. In this paper we show how such recursive structure can be used
systematically to prove isoperimetric theorems.

\end{abstract}
\maketitle

\section{Introduction}

\subsection{Background \& Motivation}

The Sierpinski gasket is a topoligical curiosity (fractal and self-similar).
According to Wikipedia, "It is named after the Polish mathematician Wac\l aw
Sierpi\'{n}ski but appeared as a decorative pattern many centuries prior to
the work of Sierpi\'{n}ski". To construct it recursively,

\begin{description}
\item[0] Start with an equilateral triangle of side $1$ (the interior as well
as the boundary).

\item[n+1] After round $n\geq0$, there remain $3^{n}$ congruent equilateral
triangles of side $1/2^{n}$. In each triangle subdivide each edge at its
midpoint and connect the midpoints to make $4$ triangles of side $1/2^{n+1}.$
Remove the interior of the central one, leaving its boundary. The
\textit{Sierpinski gasket }is the limit (set) of this process as
$n\rightarrow\infty$ (see Wikipedia).
\end{description}

The \textit{Sierpinski Gasket Graph, }$SG_{n}$, is the boundary of the set
remaining after $n$ rounds of the \textit{Sierpinski gasket }construction
(above). The vertices of $SG_{n}$ are the vertices of the constituent
triangles and its edges are the edges of those triangles. $SG_{n}$ may be
defined recursively as follows: $SG_{0}=K_{3}$, the complete graph on 3
vertices. If $SG_{n}$ has been defined for $n\geq0$, then $SG_{n+1}$ may be
constructed from 3 copies of $SG_{n}$, each copy sharing one corner vertex
with each of the other two copies. In most papers on the Sierpinski gasket
graph it is denoted by $S_{n}$. We use $SG_{n}$ because in our context $S$ is
already used for several other structures.

$SP_{n}$, a 3-dimensional analog of $SG_{n}$, is proposed in \cite{W-R-R-S} as
a connection architecture for multiprocessing computers. They call it the
"Sierpinski gasket pyramid network" but it is also known as the Sierpinski
sponge. The best-known multiprocessor architecture is $Q_{n}$, the graph of
the $n$-dimensional cube. The properties of $Q_{n}$ relevant to its employment
in computer architecture have been well studied. \cite{W-R-R-S} begins the
analysis of $SP_{n}$ by showing that (among other things)

\begin{enumerate}
\item $SP_{n}$ has diameter $2^{n-1}$.

\item $SP_{n}$ has chromatic number $4$.

\item $SP_{n}$ is hamiltonian.
\end{enumerate}

In their conclusion the authors propose studying $SP_{n}$ for its "message
routing and broadcasting" properties. This paper is following up on that
suggestion.\textbf{\ }The Edge-Isoperimetric Problem (EIP) (see \cite{Har04})
is of interest for connection graphs of multiprocessing computers because it
has implictions for message routing and broadcasting. Other authors
(\cite{H-K-M-P}) had previously proposed graphs related to $SG_{n}$ \&
$SP_{n}$ for computer architecture but did not consider their EIP.

\subsection{Definitions \& Examples}

\subsubsection{Graphs}

\begin{definition}
An \textit{ordinary} \textit{graph, }$G=\left(  V,E\right)  $ consists of a
set $V$, of \textit{vertices} and a set $E\subseteq\binom{V}{2}=\left\{
\left\{  v,w\right\}  :v,w\in V,v\neq w\right\}  ,$, of pairs of vertices
called \textit{edges}.
\end{definition}

\begin{example}
\bigskip$K_{n}$, \textit{the complete graph on }$n$ vertices has $V_{K_{n}%
}=\left\{  0,1,2,...,n-1\right\}  $ and $E_{K_{n}}=\binom{V_{K_{n}}}{2} $.
\end{example}

\begin{example}
The (disjunctive) product, $K_{m}\times K_{m}\times...\times K_{m}=K_{m}^{n} $
is called the Hamming graph. $V_{K_{n}^{n}}=\left\{  0,1,2,...,n-1\right\}
^{n}$. Two vertices ($n$-tuples of vertices of $K_{m}$) have an edge between
them if they differ in exactly one coordinate (\textit{i.e. }are at Hamming
distance $1$). Note that $K_{2}^{n}=Q_{n}$, the graph of the $n$-dimensionsal cube.
\end{example}

\subsubsection{The Edge-Isoperimetric Problem}

The Edge-Isoperimetric Problem (EIP) is a combinatorial analog of the
classical isoperimetric problem: Given a graph, $G=\left(  V,E\right)  $ and
$S\subseteq V$,
\[
\Theta\left(  S\right)  =\left\{  \left\{  v,w\right\}  \in E:v\in S\text{
}\&\text{ }w\notin S\right\}
\]
is called the \textit{edge-boundary of }$S$. Then the $EIP$ is to calculate
$\left\vert \Theta\right\vert \left(  G;\ell\right)  =\min\left\{  \left\vert
\Theta\left(  S\right)  \right\vert :S\subseteq V_{G}\&\left\vert S\right\vert
=\ell\right\}  $ for every integer $\ell$, $0\leq\ell\leq\left\vert
V\right\vert $, and identify sets that achieve the minimum. The function
$\left\vert \Theta\right\vert \left(  G;\ell\right)  $ is called the
(edge-)isoperimetric profile of $G$.

\begin{example}
For $K_{m}$, \textit{the complete graph on }$m$ vertices, any $S\subseteq
V_{K_{m}}$ with $\left\vert S\right\vert =\ell$ has $\left\vert \Theta\left(
S\right)  \right\vert =\ell\left(  m-\ell\right)  $. Thus every $\ell$-set is
a solution of the EIP for $K_{m}$. Also, its isoperimetric profile is
$\left\vert \Theta\right\vert \left(  K_{m};\ell\right)  =\ell\left(
m-\ell\right)  $.
\end{example}

\begin{example}
Initial $\mathbf{\ell}$-segments of $V_{K_{m}^{n}}$ in Lexicographic order,
\[
\left\{  0^{n},0^{n-1}1,...,\mathbf{\ell}_{1}\ell_{2}...\ell_{m}\right\}
\]
where $\ell=1+\sum_{i=1}^{m}\ell_{i}m^{n-i}$, are solutions of the EIP on
$K_{m}^{n}$ (proved for $m=2$ by the author in 1962 and for $m>2$ by John
Lindsay in 1963). Ching Guu \cite{Guu} pointed out that if we divide the
isoperimetric profile of the $n$-cube, $Q_{n}=K_{2}^{n}$ by $2^{n}$ we have%
\[
\left\vert \Theta\right\vert \left(  Q_{n};\ell\right)  /2^{n}=T(\ell
/2^{n})\text{,}%
\]
$T:\left[  0,1\right]  \rightarrow\left[  0,1\right]  $ being the celebrated
Takagi function (see Lagarias's survey \cite{Lag})
\end{example}

The property of having a numbering, $\eta:V\rightarrow\left\{
1,2,...,\left\vert V\right\vert \right\}  $, 1-1 \& onto, whose initial $\ell
$-segments, $\eta^{-1}\left(  \left\{  1,2,...,\ell\right\}  \right)  $, are
solutions of the EIP) is called \textit{nested solutions.}

\begin{example}
The classical isoperimetric problem in the plane has nested solutions,
concentric discs. $\eta:\mathbb{R}^{2}\mathbb{\rightarrow R}^{+}$ is defined
by $\eta\left(  x,y\right)  =\pi\left(  x^{2}+y^{2}\right)  $. For every
$a\geq0$, $\eta^{-1}\left(  \left[  0,a\right]  \right)  $ is a disc of area
$a$ and radius $r=\sqrt{\frac{a}{\pi}}$ centered at $\left(  0,0\right)  $.
The length of the boundary of the disc is $\lambda=2\pi r=2\sqrt{\pi a}.$ This
function, $\lambda\left(  a\right)  $, giving the minimum length of the
boundary of any set, $S\subseteq\mathbb{R}^{2}$, of area $a$, is the
isoperimetric profile of $\mathbb{R}^{2}$ (wrt the Euclidean metric).
\end{example}

\section{Results on $SG_{n}$}

\subsection{$SG_{1}$ and $SG_{2}$ have Nested Solutions for $EIP$}

For $n=1$ the result is trivial since $SG_{1}=K_{3}$ all numberings are
equivalent under symmetry and all $\ell$-sets achieve min$\left\{  \left\vert
\Theta\left(  S\right)  \right\vert :S\subseteq V_{K_{3}},\left\vert
S\right\vert =\ell\right\}  $.

For $n=2$ we apply stabilization to simplify the problem (See \cite{Har04},
Chapter 3): Figure 1 shows a diagram of $SG_{2}$ with basic reflections
\textbf{R}$_{0}$, \textbf{R}$_{1}$.%
\[%
{\parbox[b]{2.8928in}{\begin{center}
\includegraphics[
trim=0.401049in 2.018427in 0.406997in 2.696372in,
height=2.3722in,
width=2.8928in
]%
{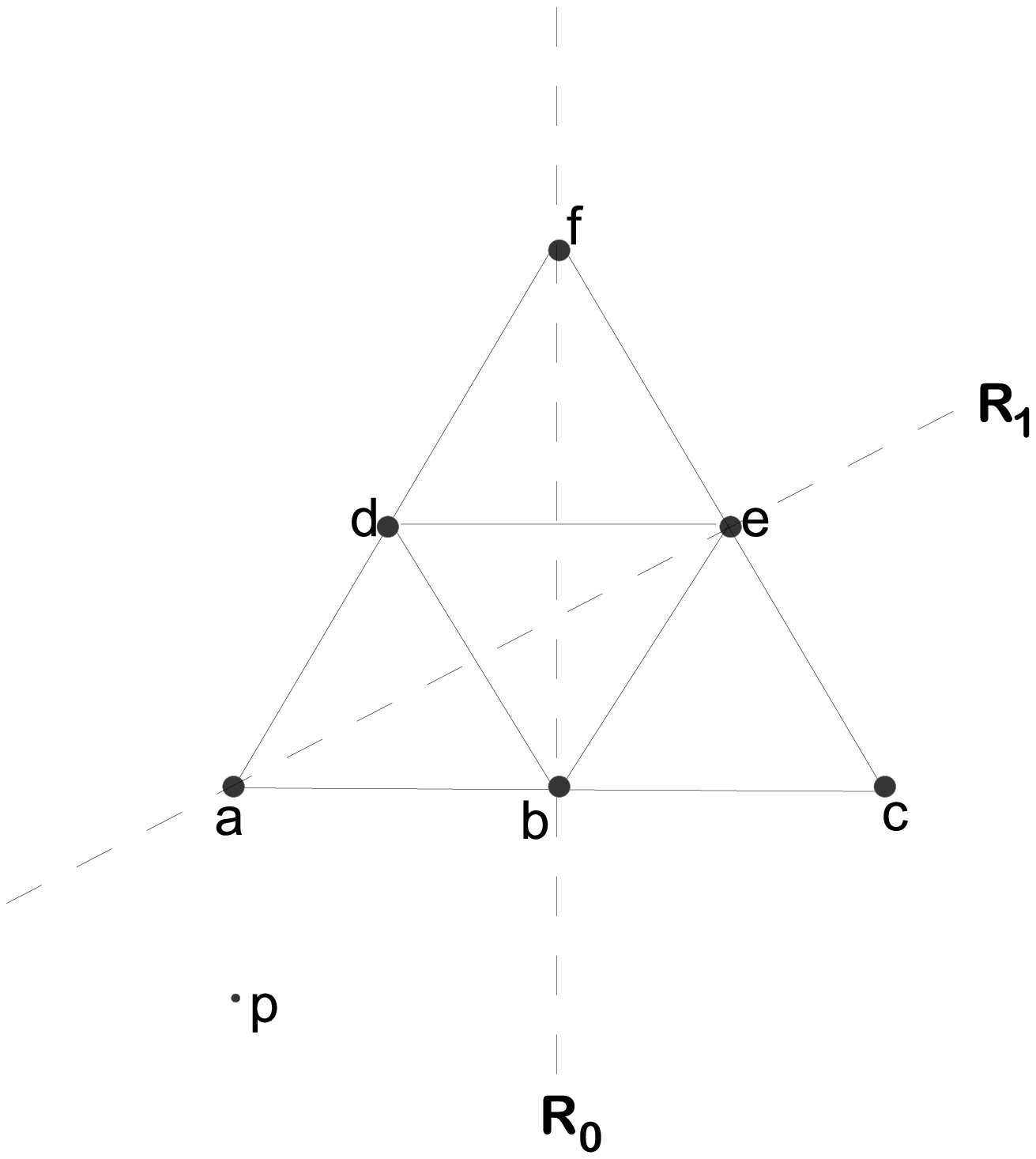}%
\\
Figure 1-$SG_2$ with basic reflections $\textbf{R}_0$, $\textbf{R}_1$
\end{center}}}%
\]

Vertices $a,b$ lie in the fundamental chamber (the sextant containing the
point, p). These are the minimal elements of the components of \ the
stabilization-order, $\mathcal{S}$-$\mathcal{O}\left(  S(2,3)\right)  $.
Coxeter theory tells us that $\mathcal{S}$-$\mathcal{O}\left(  S(n,m)\right)
$ may be constructed recursively from its minimal elements, extending each
component from rank $r$ to rank $r+1$ by applying the adjacent transpositions,
$i\left(  i+1\right)  $ to $v$ (in rank $r)$ iff $v_{j_{0}}=i$, where
$j_{0}=\min\left\{  j:v_{j}=i\text{ or }i+1\right\}  $. The resulting vector
$v^{\prime}$ ($v$ with entries $i$ and $i+1$ transposed) is then in rank
$r+1$. The Hasse diagram of $\mathcal{S}$-$\mathcal{O}\left(  S(2,3)\right)  $
is in Figure 2.%
\[%
{\parbox[b]{2.5425in}{\begin{center}
\includegraphics[
trim=1.072296in 2.961607in 0.678895in 2.696372in,
height=2.0219in,
width=2.5425in
]%
{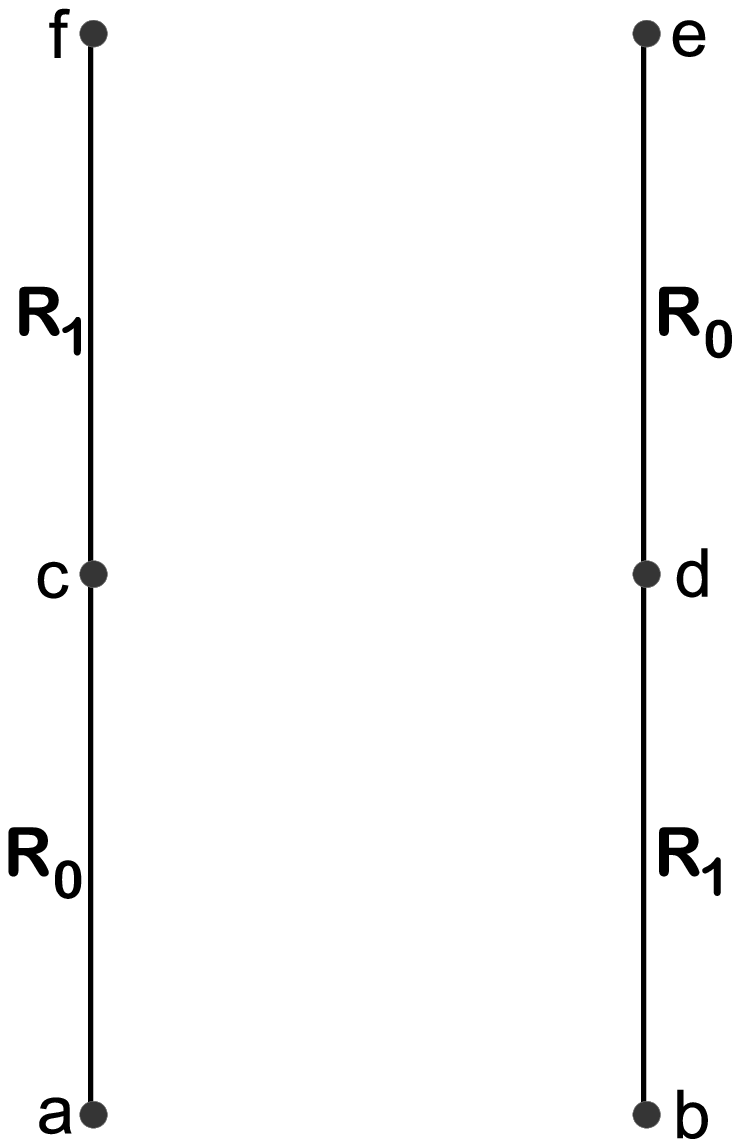}%
\\
Figure 2-The stabilization order of $SG_2$
\end{center}}}%
\]

An \textit{ideal}, $\iota$, of a partially ordered set (poset), $\mathcal{P}$,
is a subset of the poset that is downwardly closed. \textit{I.e}. if $x\leq y$
\& $y\in\iota$ then $x\in\iota$.The \textit{ideal transform, }$\mathfrak{I}%
(\mathcal{P})$, of a poset, $\mathcal{P}$, is the set of all ideals of
$\mathcal{P}$, partially ordered by $\subseteq$. The \textit{derived network
}of a $StOp$-order is the Hasse diagram of \ the ideal transform of the
$StOp$-order, with weight $\left\vert \Theta\left(  S\right)  \right\vert $
for each ideal, $S$. The derived network of the stabilization-order of
$SG_{2}$ is shown in Figure 3.%
\[%
{\parbox[b]{2.8919in}{\begin{center}
\includegraphics[
trim=0.538697in 1.884159in 0.272747in 2.022829in,
height=2.674in,
width=2.8919in
]%
{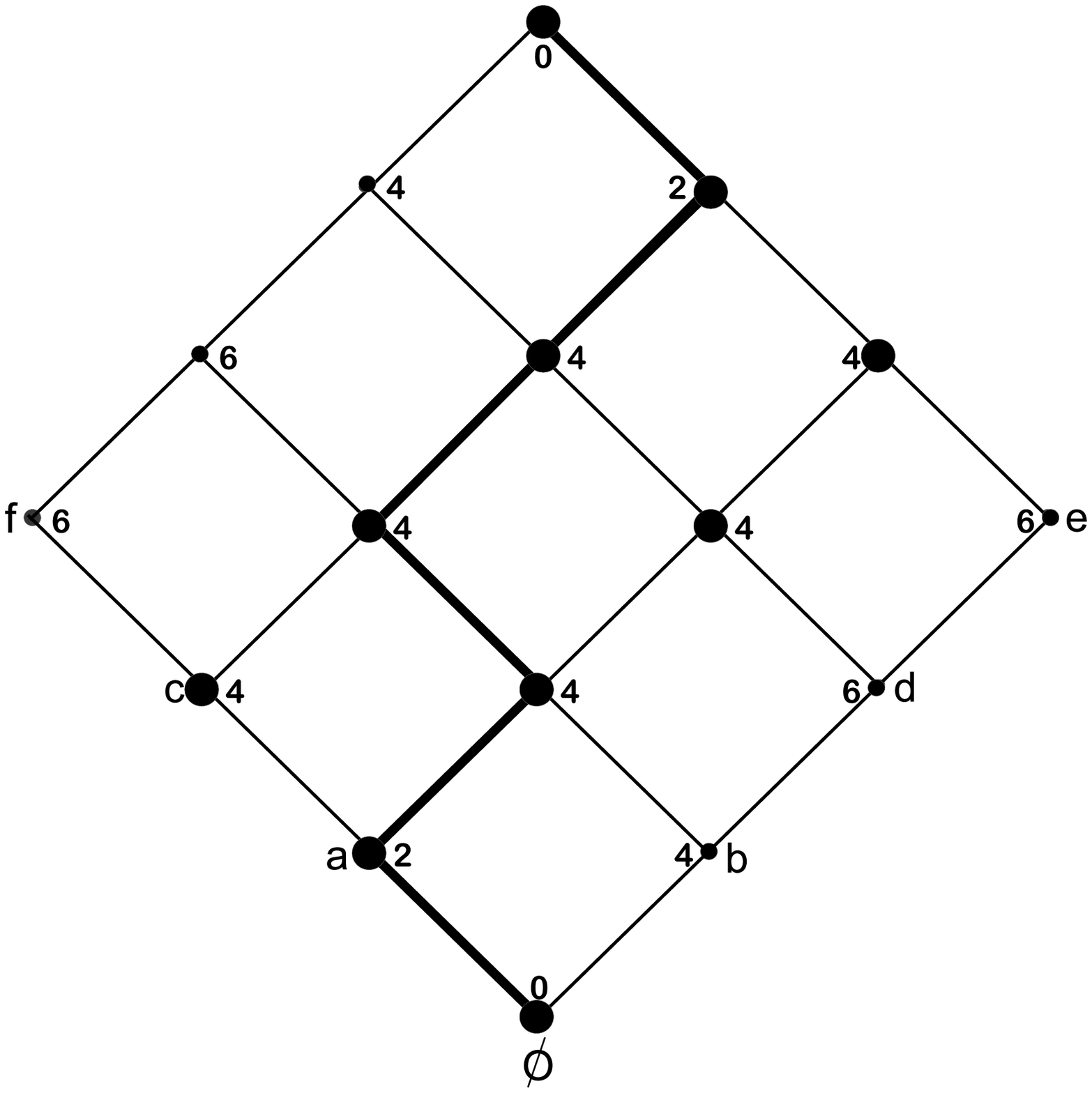}%
\\
Figure 3-The derived network of $\mathcal{S}$-$\mathcal{O}\left(  SG_2\right)
$
\end{center}}}%
\]

The large vertices represent solutions of the EIP for $SG_{2}$. The darkened
edges, tracing a path from $\emptyset$ to $\left\{  a,b,c,d,e,f\right\}  $
through solution sets, shows that $SG_{2}$ has nested solutions,%
\[
\emptyset\subset\left\{  a\right\}  \subset\left\{  a,b\right\}
\subset\left\{  a,b,c\right\}  \subset\left\{  a,b,c,d\right\}  \subset
\left\{  a,b,c,d,e\right\}  \subset\left\{  a,b,c,d,e,f\right\}  \text{.}%
\]
.

Indra Rajasingh and collaborators has shown that if $\emptyset\varsubsetneq
S\varsubsetneq V_{SG_{n}}$ then $\left\vert \Theta\left(  S\right)
\right\vert \geq2$ and this bound is sharp (achieved by the subgraphs of
$SG_{n}$ isomorphic to $SG_{n^{\prime}}$) if $\left\vert S\right\vert
=\left\vert SG_{n^{\prime}}\right\vert =\left(  3^{n^{\prime}}+3\right)  /2$,
$0\leq n^{\prime}<n$.

\subsection{$SG_{3}$ does not have Nested Solutions}

\subsubsection{A Necessary Condition}

In order for a graph, $G$, to have nested solutions for $EIP$, initial
segments of the numbering must be sequentially optimal. \textit{I.e.} the
additional vertex in each successive initial segment must minimize its
marginal contribution to edge-boundary. This observation gives us an easy way
to generate candidates for an optimal numbering. It also leads to a necessary
condition for nested solutions because the derived network (See \cite{Har04},
p. 24) of $G$ is self-dual (by complementation). The dual of a sequentially
optimal $s$-$t$ path in the derived network must also be sequentially optimal
so we can start at both ends and meet in the middle. If no sequentially
optimal paths meet in the middle then $G$ cannot have nested solutions.

The intuition about the $EIP$ on $SG_{n\text{ }}$is that the Sierpinski gasket
subgraph, $SG_{m\text{ }}$for $m<n$, has greater edge-density than $SG_{n},$
so once a numbering includes one vertex of $SG_{m\text{ }}$it should continue
numbering in that $SG_{m\text{ }}$until it is completely numbered.
Unfortunately that strategy contains the seeds of its own destruction and
$SG_{3\text{ }}$fails the above test: $SG_{2\text{ }}$is the unique (up to
isomorphism) sequential minimizer of size $6$. Its complement is then the
unique dual sequential minimizer of size $15-6=9$. However, since $SG_{2\text{
}}$has a common vertex with every other copy of itself in $SG_{3\text{ }},$
$SG_{3}-SG_{2\text{ }}$ cannot contain a copy of $SG_{2\text{ }}$and there is
no way sequentially optimal paths can meet in the middle. Another way to say
the same thing is that the function $\left\vert \Theta\right\vert \left(
G;\ell\right)  =\min\left\{  \left\vert \Theta\left(  S\right)  \right\vert
:S\subseteq V_{G},\left\vert S\right\vert =\ell\right\}  $ is symmetric about
$n/2$ (since $\left\vert \Theta\left(  S\right)  \right\vert =\left\vert
\Theta\left(  V_{G}-S\right)  \right\vert $). Therefore any numbering, $\mu$,
for which $\left\vert \Theta\left(  \mu^{-1}\left(  \boldsymbol{\ell}\right)
\right)  \right\vert $ is not symmetric about $n/2$ cannot give nested
solution for $EIP$ (even if it is sequentially optimal).

However, there are graphs closely related to the Sierpinski gasket graphs,
namely the extended Sierpinski graphs, which pass this test for nested
solutions for $EIP$.

\section{Generalized \& Expanded Sierpinski Graphs}

The \textit{generalized \& expanded Sierpinski graph, }$S(n,m)$, $n\geq1$,
$m\geq2$, was defined in 1944 by Scorer, Grundy and Smith \cite{S-G-S}:
$V_{S(n,m)}=\left\{  0,1,...,m-1\right\}  ^{n}$. For $\left\{  u,v\right\}
\in\binom{V}{2}$, $\left\{  u,v\right\}  \in E_{S(n,m)}$ iff $\exists
h\in\left\{  1,2,...,n\right\}  $ such that following 3 conditions hold:

\begin{enumerate}
\item $u_{i}=v_{i}$ for $i=1,2,...h-1$;

\item $u_{h}\neq v_{h}$; and

\item $u_{j}=v_{h}$ and $v_{j}=u_{h}$ for $j=h+1,...,n$.
\end{enumerate}

The motivation for defining $S(n,m)$ was that $S(n,3)$ is (isomorphic to) the
graph of the $3$-peg Towers of Hanoi puzzle with $n$ disks \cite{S-G-S}.
Scorer, Grundy and Smith pointed out that $SG_{n}$ is a quotient of $S(n,3)$
where every edge of $S(n,3)$ not contained in a triangle ($K_{3}$) is
contracted to a vertex. Also, the Sierpinski sponge, $SP_{n}=S[n,4]$ is a
similar quotient of $S(n,4)$. Jakovac \cite{Jak} generalized the construction
to $S[n,m]$, the quotient of $S(n,m)$ in which every edge of $S(n,m)$ not
contained in a $K_{3}$ is contracted to a vertex. He showed that $S[n,m]$ is
hamiltonian and its chromatic number is $m$.

\subsection{Structure of $S(n,m)$}

\subsubsection{The Basics}

$\left\vert V_{S(n,m)}\right\vert =m^{n}$. All $v\in V_{S(n,m)}$ have $m-1$
"interior" neighbors. These are the $n$-tuples that agree with $v$ in all
coordinates except the $n^{th}$ (the case $h=n$ in the definition of
$S(n,m)$). If $v\neq i^{n}$ then $v$ has one other ("exterior") neighbor: If
$v\neq i^{n}$ $\exists h,1\leq h<n$, such that $v_{h}\neq v_{h+1}%
=v_{h+2}=...=v_{n}$ and by definition the exterior neighbor of $v$ is
$u=v_{1}v_{2}...v_{h-1}v_{h+1}v_{h}v_{h}...v_{h}$ (note that this relationship
between $u$ and $v$ is symmetric). Thus $i^{n}$, with $i=0,1,2,...,m-1$, has
degree $m-1$ and every other vertex has degree $m$. Summing the degrees of all
vertices we get $m\left(  m-1\right)  +\left(  m^{n}-m\right)  m=$ $m^{n+1}%
-m$. Since each edge is incident to two vertices, $\left\vert E_{S(n,m)}%
\right\vert =\left(  m^{n+1}-m\right)  /2$. $K_{m}^{n}$ also has $m^{n}$
vertices (we take $V_{K_{m}^{n}}$ to be the same set, $\left\{
0,1,...,m-1\right\}  ^{n}$) and $m^{n}\left(  m-1\right)  n/2$ edges. Thus the
density of edges of $S(n,m)$ relative to $K_{m}^{n}$ is $\left(  \left(
m^{n+1}-m\right)  /2\right)  /\left(  m^{n}\left(  m-1\right)  n/2\right)
=1/n$ which is decreasing in $n$.

The vertices that agree in all except the last coordinate induce a complete
subgraph, $K_{m}$. These $K_{m}s$ are maximal, nonoverlapping and contain all
the vertices of $S(n,m)$. There are $m^{n-1}$ of them, constituting a $K_{m}%
$-decomposition of $S(n,m)$. Since any vertex is incident to at most one
exterior edge, any triangle ($K_{3}$) must contain at least two internal
edges. But then the third edge would also be internal to the same $K_{m}$, so
this $K_{m}$-decomposition is unique.

The vertices $i^{n}$, for $i=0,1,2,...,m-1$, are called \textit{corner}
vertices of $S(n,m)$. We can use them to charactize $S(n,m)$ (up to
isomorphism) recursively: Let $S(n,m)|_{v_{1}=i}$ be the subgraph of $S(n,m)$
induced by the vertices whose first coordinate is $i$. It is easy to see that
$S(n,m)|_{v_{1}=i}\simeq S(n-1,m)$. Again, these copies of $S(n-1,m)$
partition the vertices of $S(n,m)$. The edges of $S(n,m)$ not induced by
$S(n,m)|_{v_{1}=i}$ for some $i$, connect a corner of $S(n,m)|_{v_{1}=i}$ to a
corner of $S(n,m)|_{v_{1}=j}$, $i\neq j$. The rule for such a connection is
that $u\in S(n,m)|_{u_{1}=i}$ is connected to $v\in S(n,m)|_{v_{1}=j}$ iff
$u=ij^{n-1}$ and $v=ji^{n-1}$. The initial step of the recursion is to take
$S(1,m)=K_{m}$ with $V_{K_{m}}=\left\{  0,1,2,...,m-1\right\}  $. Given
$S(n-1,m),$ $n\geq2$, we construct $S(n,m)$ from $\left\{
0,1,...,m-1\right\}  \times S(n-1,m)$ by connecting $ij^{n-1}$ to $ji^{n-1}$
($\forall i\neq j$) as above.

\begin{theorem}
The symmetry group of $S(n,m)$ is $\mathcal{S}_{m}$, the symmetric group on
$m$ generators. $\mathcal{S}_{m}$ acts on the coordinates of $V_{S(n,m)}%
=\left\{  0,1,2,...,m-1\right\}  ^{n}$, \textit{i.e.}
\[
\pi\left(  v_{1},v_{2},...,v_{n},\right)  =\left(  \pi\left(  v_{1}\right)
,\pi\left(  v_{2}\right)  ,...,\pi\left(  v_{n}\right)  \right)  \text{.}%
\]

\end{theorem}

For a proof see \cite{Har15'} (Theorem 1) or \cite{H-K-M-P} (Theorem 4.14).

\subsubsection{Recursive Definition of $S(n,m)$}

The Sierpinski graph, $S(n,m)$, has been defined analytically at the beginning
of Section 3. $S(n,m)$ may also be characterized recursively: $S(1,m)=K_{m}$
and given $S(n-1,m)$ for $n-1\geq1$,
\[
V_{S(n,m)}=\left\{  0,1,...,m-1\right\}  \times V_{S(n-1,m)}%
\]
and
\[
E_{S(n,m)}=\left\{  0,1,...,m-1\right\}  \times E_{S(n,m)}+\left\{  \left\{
ij^{n-1},ji^{n-1}\right\}  :0\leq i,j\leq m-1,i\neq j\right\}  .
\]
The key property here is that the edge, $\left\{  ij^{n-1},ji^{n-1}\right\}  $
connects $\left(  i,S(n-1,m)\right)  $ to $\left(  j,S(n-1,m)\right)  $ at
unique corner vertices $ij^{n-1},ji^{n-1}$. That is to say every copy of
$S(n-1,m)$ is connected to every other copy and the edges form a complete
matching. However, any such correspondence between vertices of $\left(
i,S(n-1,m)\right)  $ for $i=0,1,...,m-1$ determines a graph isomorphic to
$S(n,m)$. Reducing the relationship to its essence, $\left\{  \left\{
ij,ji\right\}  :0\leq i,j\leq m-1,i\neq j\right\}  $ is a complete matching of
$K_{m,m}-\left\{  i^{2}:0\leq i<m\right\}  $. Actually, any complete matching
of $K_{m,m}$ will determine a graph isomorphic to $S(n,m)$. Another such
correspondence is
\[
\left\{  \left\{  ik,jk\right\}  :0\leq i,j\leq m-1,i\neq
j,i+j=k(\operatorname{mod}m)\right\}  .
\]
In \cite{Har15'} it is shown that with the latter coordinates, $S(n,m)$ is a
subgraph of $K_{m}^{n}$.

\subsubsection{Linear Coordinates for $S(n,m)$}

If $v\in V_{S(n,m)}$, so $v=\left(  v_{1},v_{2},...,v_{n}\right)  $ where
$v_{i}\in\left\{  0,1,...,m-1\right\}  ,$ then its representation in
$\mathbb{R}^{m}$ is
\[
y\left(  v\right)  =\left(  \sum_{v_{i}=0}2^{i},\sum_{v_{i}=1}2^{i}%
,...,\sum_{v_{i}=m-1}2^{i}\right)  \text{.}%
\]
Note that%
\begin{align*}
\sum_{j=1}^{n}y_{j}\left(  v\right)   &  =\sum_{i=0}^{m-1}2^{i}\\
&  =2^{n}-1\text{.}%
\end{align*}
So the vertices of $S(n,m)$ are actually lying in the hyperplane,
\[
\sum_{j=1}^{n}y_{j}=2^{n}-1\text{.}%
\]
The coordinates of $y\left(  v\right)  $ are integral, non-negative and
characterized by the fact that each power of $2$ ($2^{i}$ for $0\leq i<m$)
occurs in exactly one of the base two representations of the $y_{j}s$. Also,
two vertices are connected by an edge iff the Euclidean distance between them
is $1$.

\subsection{The $EIP$ on $S(n,m)$}

The intuition (Section 2.2.1, second paragraph) suggesting that $SG_{n}%
=S[n,3]$ might have nested solutions for $EIP$ applies more generally to
$S[n,m]$. However, the same counterexample works for $S[n,m]$ except when
$n=1$ (all $m$) and $n=2,m=3$. On the other hand, since the density of edges
in $S(n,m)$ decreases with $n$ (see Section 3.1.1), the intuition also applies
to $S(n,m)$. It suggests that for $\ell=\left(  3^{n}-1\right)  /2$, the
disjoint union $\left\{  0\right\}  \times S(n-1,3)+\left\{  10\right\}
\times S(n-2,3)+...+\left\{  1^{n-1}0\right\}  $, which is sequentially
optimal, should be optimal (minimize $\left\vert \Theta\left(  S\right)
\right\vert $ for $S$ having cardinality $3^{n-1}+3^{n-2}+...+1=\frac{3^{n}%
-1}{3-1}=$ $\frac{3^{n}-1}{2}$). The transposition $02$ of $\left\{
0,1,2\right\}  $ gives another sequentially optimal set, $\left\{  2\right\}
\times S(n-1,3)+\left\{  12\right\}  \times S(n-2,3)+...+\left\{
1^{n-1}2\right\}  $. Both of these sets are optimally extended by adding
$1^{n}$. The complement of the first set is the extension of the second (and
\textit{vice versa}). Thus the two sequentially optimal paths meet in the
middle satisfying our necessary condition (Section 2.2.1). A sequentially
optimal numbering for $S(n,m)$ is given by lexicographic order on $\left\{
0,1,...,m-1\right\}  ^{n}$,
\[
\eta\left(  v\right)  =Lex(v)=1+\sum_{i=1}^{n}v_{i}m^{n-i}\text{.}%
\]
Note that $\sum_{i=1}^{n}v_{i}m^{n-i}$ is the base $m$ representation of an
integer between $0$ and $m^{n}-1$.

The analog of the theorem of Rajasingh\textit{\ et al }(see the last paragraph
of Section 2.11)\textit{\ }holds for $S(n,m)$.

\begin{theorem}
If $\emptyset\varsubsetneq S\varsubsetneq V_{S(n,m)}$ then $\left\vert
\Theta\left(  S\right)  \right\vert \geq m-1$ and this bound is sharp
(achieved by the subgraphs of $S(n,m)$ isomorphic to $S(n^{\prime},m)$) if
$\left\vert S\right\vert =m^{n^{\prime}}$, $0\leq n^{\prime}<n$.

\begin{proof}
This follows from the fact that any two distinct points in $S(n,m)$ may be
connected by $m-1$ disjoint paths. So if $v\in S$ and $w\notin S$ there must
be $m-1$ disjoint paths from $v$ to $w$. Each such path will have a first
vertex, $w^{\prime}$, not in $S$ and the edge, $\left\{  v^{\prime},w^{\prime
}\right\}  $ from the previous vertex, $v^{\prime}\in S$, to $w^{\prime}\notin
S$, will be in the edge-boundary of $S$. Since the paths are disjoint,
$\left\vert \Theta\left(  S\right)  \right\vert \geq m-1$. Also,
$Lex^{-1}(\left\{  1,2,...,m^{n^{\prime}}\right\}  )\simeq S\left(  n^{\prime
},m\right)  $ so those initial segments are solutions.
\end{proof}
\end{theorem}

However, being sequentially optimal and satisfying the necessary condition
does not constitute a proof that initial $\ell$-segments of $Lex$ actually
minimize $\left\vert \Theta\left(  S\right)  \right\vert $ for all
cardinalities $\ell$. It could still be possibile that some $\ell\neq
m^{n^{\prime}}$ there is a strange collection of $\ell$ vertices in $S(n,m) $
that has smaller edge-boundary than $Lex^{-1}\left(  \boldsymbol{\ }\left\{
1,2,...,\ell\right\}  \right)  $.

\begin{conjecture}
The generalized \& expanded Serpinski graph, $S(n,m)$, has nested solutions
for the $EIP$. $Lex$ (lexicographic order on $V_{S(n,m)}=\left\{
0,1,...,m-1\right\}  ^{n}$) is not only sequentially optimal but initial
$\ell$-segments of $Lex$ minimize $\left\vert \Theta\left(  S\right)
\right\vert $ for $S$ having cardinality $\ell$. In other words, $\left\vert
\Theta\left(  Lex^{-1}\left(  \left\{  0,1,...,\ell\right\}  \right)  \right)
\right\vert =\left\vert \Theta\right\vert \left(  S(n,m);\ell\right)  $
$\forall\ell,0\leq\ell\leq m^{n}$.
\end{conjecture}

\subsubsection{The Case $m=2$, all $n$}

In \cite{Har15'} It is shown that $S(n,2)$ is a path of length $2^{n}-1$. The
endpoints are the corner vertice $0^{n}$ and $1^{n}$. In between, the vertices
($V_{S(n,2)}=\left\{  0,1\right\}  ^{n}$) appear in lexicographic order. So,
starting from $0^{n}$, the initial segment, $Lex^{-1}\left(  \left\{
0,1,...,\ell\right\}  \right)  $, of cardinality $\ell$, is a solution of the
$EIP$ on $S(n,2)$ for $\left\vert S\right\vert =\ell$. This proves Conjecture
1 for $m=2$.

\section{Preliminaries to Proving Conjecture 1 for $m>2$}

\subsection{Our Strategy}

Our basic logical strategy for proving that $S(n,m)$ has Lex-nested solutions
is induction on $n$. The inductive step, reducing the conjecture for
$S(n+1,m)$ to that for $S(n,m)$ is accomplished through a series of morphisms
for $EIP$ called "Steiner operations".

\subsection{Steiner Operations on $S(n,m)$}

A \textit{Steiner operation} on a graph, $G=(V,E)$, is a function, $StOp:$
$2^{V}\rightarrow2^{V}$, mapping subsets of $V$ to subsets of $V$ such that

\begin{enumerate}
\item \bigskip$\forall S\subseteq V$, $\left\vert StOp\left(  S\right)
\right\vert =\left\vert S\right\vert $,

\item \bigskip$\forall S\subseteq V$, $\left\vert \Theta\left(  StOp\left(
S\right)  \right)  \right\vert \leq\left\vert \Theta\left(  S\right)
\right\vert $ and

\item \bigskip$\forall S\subseteq T\subseteq V$, $StOp\left(  S\right)
\subseteq StOp\left(  T\right)  $.
\end{enumerate}

This definition is taken from Chapter 2 of \cite{Har04} where the theory of
StOps is developed with applications. Properties 1 \& 2 are essential for a
mapping to preserve the $EIP$. If they hold we can say that the $StOp$
represents a simplification of the $EIP$ on $G$ since we need only consider
sets in the range of the $StOp$. However, to make that simplification
effective, we must be able to pick out those subsets of $V$ that are in the
range of $StOp$ and do it efficiently. Originally this was done for each
$StOp$ considered by finding a partial order on $V$ (the $StOp$-order,
$\mathcal{S}$-$\mathcal{O}\left(  StOp\right)  $) such that $S$ is in the
range of $StOp$ iff $S$ is an ideal (If $x\leq_{\mathcal{S}\text{-}%
\mathcal{O}}y\in S$ then $x\in S$) of $\mathcal{S}$-$\mathcal{O}\left(
StOp\right)  $. More recently \cite{Har11} we showed that Properties 1 \& 3
imply that every $StOp$ has such a $StOp$-order, $\mathcal{S}$-$\mathcal{O}%
\left(  StOp\right)  $ (characterizing its range). Property 3 is called
\textit{monotonicity}.

\subsection{Stabilization}

Stabilization is a $StOp$ that utilizes reflective symmetry of $G$ in a
systematic way to achieve its simplification. The theory of stabilization is
presented in \cite{Har04}, Chapters 3 \& 6. The most important fact about
stabilization (besides being a Steiner operation) is that cyclic compositions
of $Stab_{i,i+1}$ eventually become constant. We denote the resulting "limit"
as $Stab_{\infty}$.

\subsubsection{The Stabilization-Order of $S(n,m)$}

The symmetry group of $S(n,m)$ is $\mathfrak{S}_{m}$, the symmetric group on
$\left\{  0,1,...,m-1\right\}  $, acting on the components of $V_{S(n,m)}%
=\left\{  0,1,...,m-1\right\}  ^{n}$ (See \cite{Har15'} or \cite{H-K-M-P},
Theorem 4.14). $\mathfrak{S}_{m}$, with the generating set $W=\left\{
01,12,...,\left(  m-2\right)  \left(  m-1\right)  \right\}  $ is a Coxeter
group, \textit{i.e. }it is generated by elements of order 2 (See \cite{B-B}
for Coxeter theory). When we consider $S(n,m)$ as embedded in $\mathbb{R}^{m}$
(see Section 3.1.3), the transpositions, $ij$, correspond to reflections so
they define stabilization operations (see Section 3.2.4 of \cite{Har04}). In
particular, the fixed hyperplanes of the reflections induced by the (adjacent)
transpositions of $W$ surround the fundamental chamber, $C_{0}=\left\{
y\in\mathbb{R}^{m\text{ }}:y_{1}\geq y_{2}\geq...\geq y_{m}\geq0\right\}  $.
There are $m!$ chambers altogether, one for each member of $\mathfrak{S}_{m}$.

The stabilization-order, $\mathcal{S}$-$\mathcal{O}\left(  S(n,m)\right)  $,
is a disjoint union of components, each of which has a unique minimum element
in $C_{0}^{\prime}=y^{-1}\left(  C_{0}\right)  $ (Theorems 5.3-.5 of
\cite{Har04}).

\begin{example}
The corner vertex, $0^{n}$, is in $C_{0}^{\prime}$. Its connected component
(of $\mathcal{S}$-$\mathcal{O}\left(  S(n,m)\right)  $) is $\left\{
0^{n}\lessdot1^{n}\lessdot,...,\lessdot\left(  m-1\right)  ^{n}\right\}  $.
"$\lessdot$" represents the covering relation in $\mathcal{S}$-$\mathcal{O}$
(given by a basic transposition, $i(i+1)$).
\end{example}

\begin{example}
The vertex, $01^{n-1}$, is in $C_{0}^{\prime}$. For $m=3$ its component
consists of
\[
\left\{  01^{n-1},10^{n-1},02^{n-1},20^{n-1},12^{n-1},21^{n-1}\right\}
\]
with basic covering relations (given by adjacent transpositions, $i\left(
i+1\right)  $) $01^{n-1}\lessdot_{01}10^{n-1}\lessdot_{12}20^{n-1}%
\lessdot_{01}21^{n-1}$ and $01^{n-1}\lessdot_{12}02^{n-1}\lessdot_{01}%
12^{n-1}\lessdot_{12}21^{n-1}$. Also it has nonbasic covering relations
$10^{n-1}\lessdot_{02}12^{n-1}$ and $02^{n-1}\lessdot_{02}20^{n-1}$. Its Hasse
diagram is shown in Figure 4. The heavy lines represent actual edges of
$S(n,3)$. Note that the component of $01^{n-1}$ is also $21^{n-1}\downarrow$,
the ideal (of $\mathcal{S}$-$\mathcal{O}S(n,3)$) generated by (below)
$21^{n-1}$.%
\[%
{\parbox[b]{2.8928in}{\begin{center}
\includegraphics[
trim=0.400199in 3.231244in 0.407846in 3.234546in,
height=1.7201in,
width=2.8928in
]%
{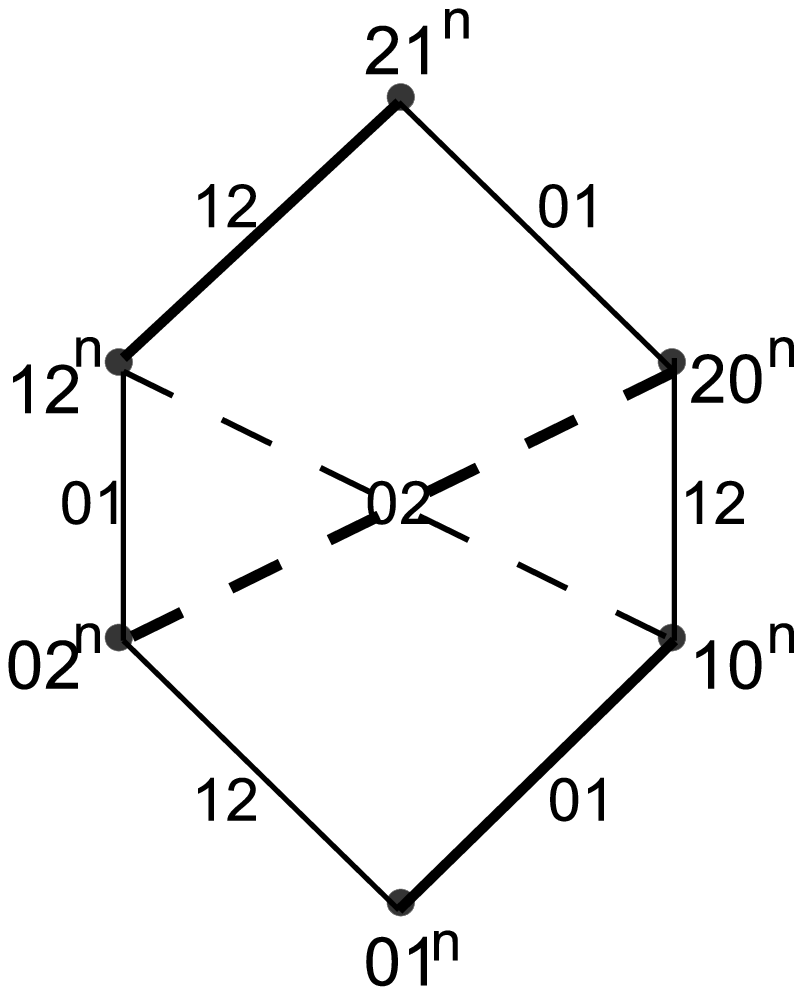}%
\\
Figure 4-The component of 01$^n$ for $m=3$
\end{center}}}%
\]

\end{example}

\begin{example}
For $m=4$ the corresponding Hasse diagram is Figure 5.%
\[%
{\parbox[b]{2.9948in}{\begin{center}
\includegraphics[
trim=0.264250in 2.423433in 0.273597in 2.291366in,
height=2.3722in,
width=2.9948in
]%
{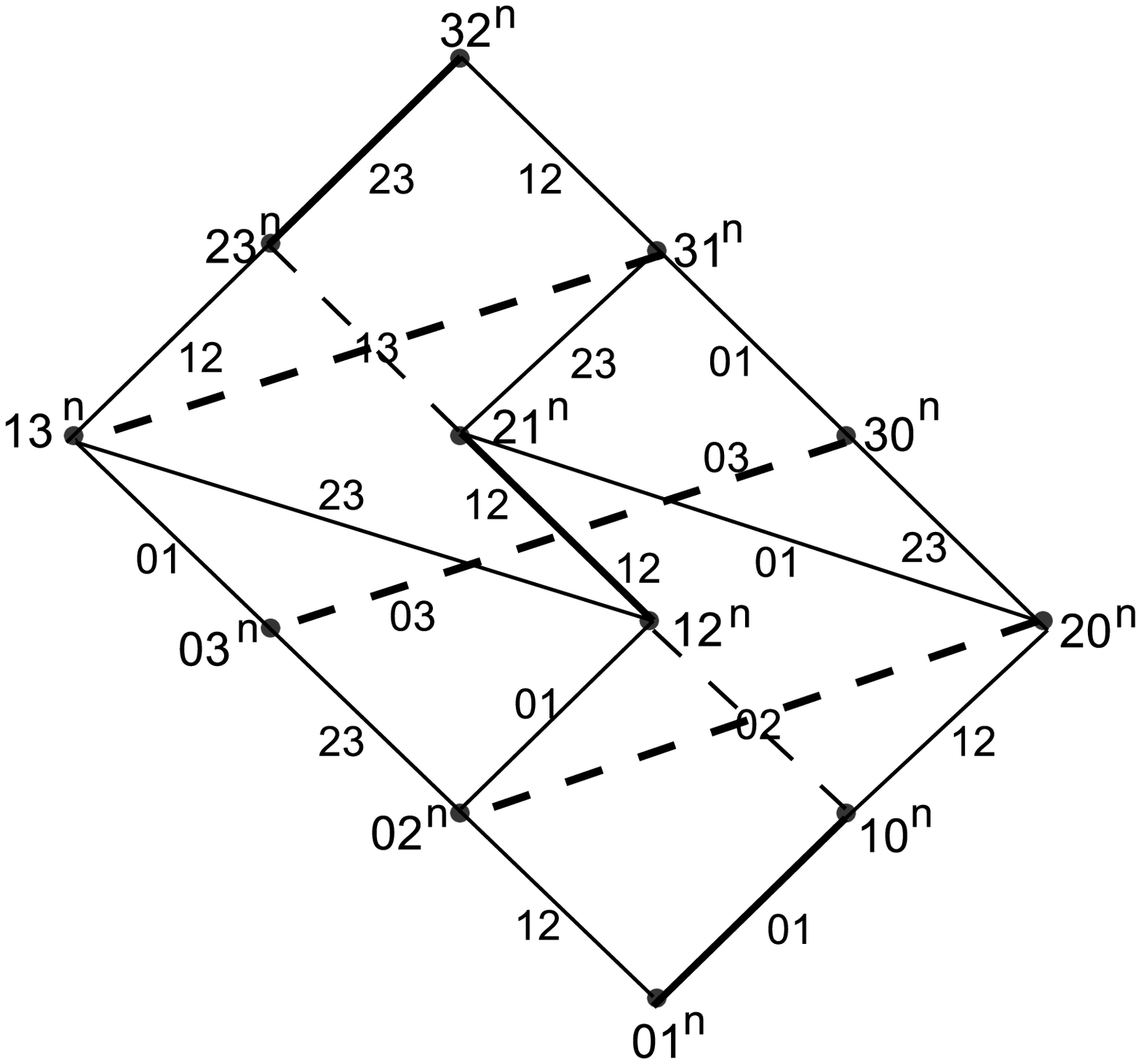}%
\\
Figure 5-The component of $01^n$ for $m=4$
\end{center}}}%
\]

\end{example}

\begin{example}
The stabilization-order of $S(2,3)$ has just two components, those given in
the Examples 6 \& 7. Compare to $\mathcal{S}$-$\mathcal{O}\left(  S\left[
2,3\right]  \right)  $ in Figure 2.
\end{example}

Recall that an \textit{ideal, }$\iota$, of a partially ordered set (poset) is
subset of the poset that is downward closed, \textit{i.e.} if $x\leq y$ \&
$y\in\iota$ then $x\in\iota$. The \textit{ideal transform, }$\mathfrak{I}%
\left(  \mathcal{P}\right)  $ of a poset $\mathcal{P}$, is the set of all
ideals of $\mathcal{P}$, partially ordered by containment.

\begin{example}
Figure 6 shows the ideal transform of the component of $01^{n}$ for $m=3$
(Figure 4).%
\[%
{\parbox[b]{2.5927in}{\begin{center}
\includegraphics[
trim=0.804647in 1.883058in 0.810595in 1.616723in,
height=2.8253in,
width=2.5927in
]%
{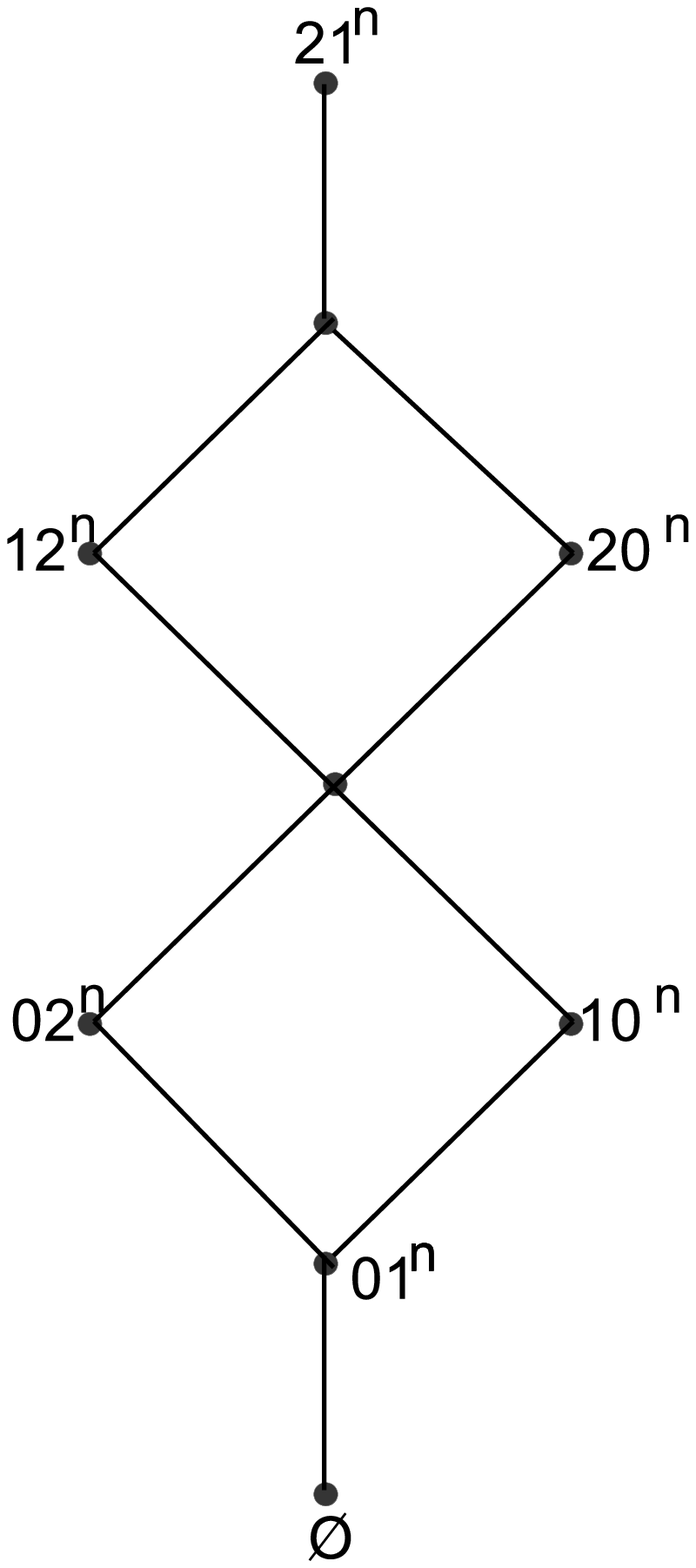}%
\\
Figure 6-The ideal transform of $21\downarrow$
\end{center}}}%
\]

\end{example}

Each point in the diagram represents an ideal, the set of elements below it.
So the (unlabeled) midpoint of the diagram represents the ideal $\left\{
01^{n},02^{n},10^{n}\right\}  $

\subsubsection{Three Key Observations}

Given \ $S\subseteq V_{S(n,m)}=\left\{  0,1,...,m-1\right\}  ^{n}$, with
$\left\vert S\right\vert =\ell$, let $\left\vert S\cap\left(  \left\{
k\right\}  \times S(n-1,m)\right)  \right\vert =\ell_{k}$, so $\sum
_{k=0}^{m-1}\ell_{k}=\ell$. Also let $\boldsymbol{\ell}\left(  S\right)
=\left(  \ell_{0},\ell_{1},...,\ell_{m-1}\right)  $ and totally order the
$m$-tuples $\boldsymbol{\ell}\left(  S\right)  $ lexicographicaly. Then

\begin{enumerate}
\item $\boldsymbol{\ell}\left(  Stab_{ij}\left(  S\right)  \right)
\geq\boldsymbol{\ell}\left(  S\right)  $.

\item If $S$ is stabilized wrt all the generators of $\mathfrak{S}_{m}$
(\textit{i.e. }$Stab_{ij}\left(  S\right)  =S$ $\forall ij$), then $\ell
_{0}\geq\ell_{1}\geq...\geq\ell_{m-1}$.

\item $\forall i,j,$if $\pi\in\mathfrak{S}_{m}$ and $Lex_{\pi}\left(
v\right)  =1+\sum_{i=1}^{n}v_{\pi\left(  i\right)  }m^{n-i}$, then
\[
Stab_{ij}\left(  Lex_{\pi}^{-1}\left(  \left\{  1,2,...,\ell\right\}  \right)
\right)  =\left\{
\begin{array}
[c]{ll}%
Lex_{\pi}^{-1}\left(  \left\{  1,2,...,\ell\right\}  \right)  & \text{if }%
\pi\left(  i\right)  <\pi\left(  j\right) \\
Lex_{\pi\circ ij}^{-1}\left(  \left\{  1,2,...,\ell\right\}  \right)  &
\text{if }\pi\left(  i\right)  >\pi\left(  j\right)
\end{array}
\right.
\]
So if we apply basic stabilizations sufficiently many times to $Lex_{\pi}%
^{-1}\left(  \left\{  1,2,...,\ell\right\}  \right)  $ we will get
$Lex_{\iota}^{-1}\left(  \left\{  1,2,...,\ell\right\}  \right)  $, $\iota$
being the identity permutation and $Lex_{\iota}=Lex$, the standard
Lexicographic order. $Lex^{-1}\left(  \left\{  1,2,...,\ell\right\}  \right)
$ is stable, \textit{i.e. }$\forall i,j,$ $Stab_{ij}\left(  Lex^{-1}\left(
\left\{  1,2,...,\ell\right\}  \right)  \right)  =Lex^{-1}\left(  \left\{
1,2,...,\ell\right\}  \right)  $.
\end{enumerate}

\subsection{Compression}

In \cite{Har04} compression is a Steiner operation on a product of graphs,
$G\times H$, based on at least one of the factors (say $G$) having nested
solutions (See the Appendix (Section 7.2) for more about compression on
$G\times H$). The fact that $Comp_{\mu,G\times H}$ is a $StOp$ is key to
proving nested solutions for a product, such as $K_{m}^{n}$, recursively.
Compression has been the single most powerful tool in proving combinatorial
isoperimetric theorems. In \cite{Har15} we were able to extend compression to
a self-similar structure (not a product, but still having nested solutions)
and thereby induct on the depth of self-similarity. Applying the strategy to
$S(n,m)$ has been problematic, however, in that we were unable to incorporate
Property 3 (monotonicity) into our extended definition of compression for
$S(n,m)$.

Our ultimate goal is to prove Conjecture 1, that $\forall n,m,\ell$,
$\left\vert \Theta\right\vert \left(  S(n,m);\ell\right)  =\left\vert
\Theta\left(  Lex^{-1}\left\{  1,2,...,\ell\right\}  \right)  \right\vert $.
Now suppose Conjecture 1 is true for some $n>1$ and we wish to prove it for
$n+1$ using some form of compression: Then $\left\{  h\right\}  \times S(n,m)$
is a sub-Sierpinski graph of $S(n+1,m)$. $\left\{  h\right\}  \times S(n,m)$
has $Lex$ order on it, induced by restricting $Lex$ on $S(n+1,m)$ to $\left\{
h\right\}  \times S(n,m)$, which is the same as $Lex$ order on $S(n,m)$. That
is, the numbers are not the same but their relative order is.

If $S\subseteq V_{S(n+1,m)}$ is stabilized, $\left\vert S\right\vert =\ell$
and $\left\vert S\cap V_{\left\{  j\right\}  \times S(n,m)}\right\vert
=\ell_{j}$, then $\ell_{0}\geq\ell_{1}\geq...\geq\ell_{m-1}\geq0$ and
$\sum_{j=0}^{m-1}\ell_{j}=\ell$. Emulating the definition of $Comp_{\mu
,G\times H}$, we should define the compression operation, $Comp_{Lex,j}:$
$2^{V_{S(n+1,m)}}\rightarrow$ $2^{V_{S(n+1,m)}}$, by $Comp_{Lex,j}\left(
S\right)  =S-\left(  S\cap V_{\left\{  j\right\}  \times S(n,m)}\right)
+Lex_{j}^{-1}\left\{  1,2,...,\ell_{j}\right\}  $. That is, $Comp_{Lex,j}$
should remove those elements of $S$ in $V_{\left\{  j\right\}  \times S(n,m)}$
and replace them by the initial $\ell_{j}$-segment of $Lex_{j}$. And we then
want $Comp_{Lex,j}$ to be a Steiner operation, utilizing the fact that
$Lex_{j}^{-1}\left\{  1,2,...,\ell_{j}\right\}  $ minimizes $\left\vert
\Theta\left(  S\cap V_{j\times S(n,m)}\right)  \right\vert $. However, there
is a problem with the exterior edges of $\left\{  j\right\}  \times S(n,m)$
(the ones of the form $\left\{  ji^{n},ij^{n}\right\}  $, $i\neq j$). Those
edges can effect the contribution of $S\cap V_{\left\{  j\right\}  \times
S(n,m)}$ to $\left\vert \Theta\left(  S\right)  \right\vert $ and $\left\vert
\Theta\left(  Comp_{Lex,j}\left(  S\right)  \right)  \right\vert $ so that
$Comp_{Lex,j}$ need not satisfy Property 2 of a Steiner operation. To remedy
the situation we define $I_{j}=\left\{  i\neq j:ij^{n}\in S\right\}  $ and
reorder $\left\{  0,1,...,m-1\right\}  $ so that the members of $I_{j}$ appear
in their relative order, $i_{0}<i_{1}<...<i_{\left\vert I_{j}\right\vert -1}$,
then $i_{\left\vert I_{j}\right\vert }=j$ (note that $j^{n+1}$ is the corner
vertex in $\left\{  j\right\}  \times S(n,m)$ and not connected to any vertex
outside of $\left\{  j\right\}  \times S(n,m)$) followed by the members of
$\ K_{j}=\left\{  0,1,...,m-1\right\}  -I_{j}-\left\{  j\right\}  $ in their
relative order. Then in our ammended definition of $Comp_{Lex,j}$, $Lex_{j}$
will be $Lex_{\pi}$ order on $\left\{  j\right\}  \times S(n,m)$ wrt this
reordering, $\pi=$ $I_{j}jK_{j}$, of $\ \left\{  0,1,...,m-1\right\}  $ (See
Section 4.3.2(3) for the definition of $Lex_{\pi}$).

Now let $\left\{  j\right\}  \times S(n,m)$ be the $j+1^{st}$ row (actually
the sub-Sierpinski graph, one level down) of $S(n+1,m)$ and assume that the
first $j$ rows, numbered $0,1,...,j-1$, have already been filled in. We limit
ourselves to $S\subseteq\left\{  j\right\}  \times S(n,m)$. This means that
$\left\vert \Theta\right\vert \left(  \left\{  j\right\}  \times
S(n,m);\ell\right)  $ is defined to be $\min\left\{  \left\vert \Theta\left(
S\right)  \right\vert :S\subseteq V_{\left\{  j\right\}  \times S(n,m)}%
,\left\vert S\right\vert =\ell\right\}  $ so that $\left\vert \Theta\left(
\left\{  j\right\}  \times S(n,m);0\right)  \right\vert =j$ (it counts the
edges "cut" by $0j^{n},...,\left(  j-1\right)  j^{n}$). We could then try to
prove
\[
\forall n,m,j,\ell_{j},\left\vert \Theta\right\vert \left(  \left\{
j\right\}  \times S(n,m);\ell_{j}\right)  =\left\vert \Theta\left(
Lex_{j}^{-1}\left\{  1,2,...,\ell_{j}\right\}  \right)  \right\vert
\]
by induction. Unsuccessful attempts to do so led us to extend Conjecture 1
even further: Let
\begin{align*}
I  &  =\left\{  0,1,...,s-1\right\}  ,\\
J  &  =\left\{  s,s+1,...,s+t-1\right\}  ,\\
K  &  =\left\{  s+t,s+t+1,...,m-1\right\}  ,
\end{align*}
so $\left\vert I\right\vert =s,$ $\left\vert J\right\vert =t$ $\&$ $\left\vert
K\right\vert =m-s-t$. For $s,t\geq0$ $\&$ $s+t\leq m$ consider $S_{s,t}(n,m)$
to be the graph $S(n,m)$ with "exterior" edges attached to the corner
vertices. If $i\in I+K$ then $\left\{  v_{i},i^{n}\right\}  \in E_{S_{s,t}%
(n,m)}$ and when computing $\left\vert \Theta_{s,t}\left(  S\right)
\right\vert $ we consider $v_{i}$ to be a member of $S$ if $i\in I$ but to be
in the complement of $S$ if $i\in K$ . Vertices $j^{n}$ for $j\in J $ are
considered to be corner vertices, not incident to an "exterior" edge.

\begin{conjecture}
$\forall\ell,$ $0\leq\ell\leq m^{n},$ $\left\vert \Theta_{s,t}\right\vert
\left(  S_{s,t}(n,m);\ell\right)  =$ $\left\vert \Theta_{s,t}\left(
Lex^{-1}\left(  \ell\right)  \right)  \right\vert $ .
\end{conjecture}

Conjecture 1 is then the special case of Conjecture 2 with $I=\emptyset=K$ and
$J=\left\{  0,1,...,m-1\right\}  $. The point is that the optimal order on
$S_{s,t}(n,m)$ is independent of $I,J,K$, even though the exterior edges of
$S_{s,t}(n,m)$ vary with $s,t$. Strengthening the inductive hypothesis is a
standard strategy when proving theorems by induction, and it is what makes
compression work in this context.

Since any permutation of $\left\{  0,1,...,m-1\right\}  $ induces a symmetry
of $S(n,m)$, from the point of view of $\left\{  h\right\}  \times
S_{s,t}(n,m) $, the exterior edges whose other ends are in $S$ may be regarded
as coming from the previous ranks (renumbered $0,1,...,s^{\prime}-1 $) and the
exterior edges whose other ends are not in $S$ may be regarded as going to the
succeeding ranks (renumbered $s^{\prime}+t^{\prime},s^{\prime}+t^{\prime
}+1,...,m-1$). It is wrt this renumbering that we define $Lex_{h} $.

\begin{theorem}
\bigskip$\forall S\subseteq V_{S_{s,t}(n+1,m)}$, with $Comp_{Lex_{h}}\left(
S\right)  =S-\left(  S\cap V_{\left\{  h\right\}  \times S_{s.t}(n,m)}\right)
+Lex_{h}^{-1}\left\{  1,2,...,\ell_{h}\right\}  $,
\end{theorem}

\begin{enumerate}
\item \bigskip\ $\left\vert Comp_{Lex_{h}}\left(  S\right)  \right\vert
=\left\vert S\right\vert $ and

\item \bigskip$\left\vert \Theta_{s,t}\left(  Comp_{Lex_{h}}\left(  S\right)
\right)  \right\vert \leq\left\vert \Theta_{s,t}\left(  S\right)  \right\vert
$.
\end{enumerate}

\begin{proof}
\begin{enumerate}
\item $\left\vert Comp_{Lex_{h}}\left(  S\right)  \right\vert =\left\vert
S\right\vert -\left\vert \left(  S\cap V_{h\times S_{s,t}(n,m)}\right)
\right\vert +\left\vert Lex_{h}^{-1}\left\{  1,2,...,\ell_{h}\right\}
\right\vert $

\ \ \ \ \ \ \ \ \ \ \ \ \ \ \ \ \ \ \ \ \ \ \ $=\left\vert S\right\vert
-\ell_{h}+\ell_{h}=\left\vert S\right\vert .$

\item $\left\vert \Theta_{s,t}\left(  S\right)  \right\vert $ counts the edges
cut by $S$, some of which are interior to \linebreak$V_{\left\{  j\right\}
\times S_{s,t}(n,m)}$ and others that are exterior. $Comp_{Lex_{h}}\left(
S\right)  $ (in minimizing the number of those that are incident to
$V_{\left\{  j\right\}  \times S_{s,t}(n,m)}$ (interior and exterior) does not
increase that number and (in not changing those that are not) does not
increase that number either. Anyway, the sum of their cardinalities is not
increased so $\left\vert \Theta_{s,t}\left(  Comp_{Lex_{h}}\left(  S\right)
\right)  \right\vert \leq\left\vert \Theta_{s,t}\left(  S\right)  \right\vert
$.
\end{enumerate}
\end{proof}

Unfortunately, $Comp_{Lex_{h}}$ is not monotonic (Property 3 of a Steiner
operation) as the following example shows.

\begin{example}%
\[
\bigskip Comp_{Lex_{0}}\left(  \left\{  01^{n}\right\}  \right)  =\left\{
0^{n+1}\right\}  ,
\]
and \bigskip%
\[
Comp_{Lex_{0}}\left(  \left\{  01^{n},10^{n}\right\}  \right)  =\left\{
01^{n},10^{n}\right\}  \text{.}%
\]
So%
\[
\left\{  01^{n}\right\}  \subseteq\left\{  01^{n},10^{n}\right\}  \text{,}%
\]
but%
\[
Comp_{Lex_{0}}\left(  \left\{  01^{n}\right\}  \right)  \nsubseteq
Comp_{Lex_{0}}\left(  \left\{  01^{n},10^{n}\right\}  \right)  \text{.}%
\]

\end{example}

Thus $Comp_{Lex_{0}}$ is not a full $StOp$. This creates technical
difficulties but (as we remarked after the definition of a Steiner operation)
properties $1)$ \& $2)$ are the essential ones for solving isoperimetric
problems and technical difficulties can be overcome. We call $Comp_{Lex_{h}}$
a \textit{nonmonotone Steiner operation}\textbf{.}

As noted in the previous section, the cyclic composition of stabilizations are
eventually constant (defining $Stab_{\infty}$). In fact the cyclic
compositions of any finite set of Steiner operations, as long as they are
\textit{consistent} (their $StOp$-$Orders$ have a common total extension) will
eventually become constant functions. This "limit" of cyclic compositions is,
in category theory, the \textit{pushout} of the constituent Steiner
operations. This also applies to compression in \cite{Har04} because in
\cite{Har04}, compression is defined only for products and is a full Steiner
operation. But what about our nonmonotone compressions? Do cyclic compositions
of $Comp_{Lex_{0}}$, $Comp_{Lex_{1}}$, ...,$Comp_{Lex_{m-1}}$ eventually
become constant? This is the "technical problem" referred to earlier as caused
by the nonmonotonicity of our extended compression operation. Since the
domain, $2^{V_{G}}$, (which is also the codomain) of any $StOp$ is finite,
cyclic compositions must eventually cycle, so the question is whether that
cycle is necessarily of length $1$?

For $S\subseteq V_{S_{s,t}(n,m)}$, let
\[
\tau\left(  S;ij^{n}\right)  =\left\{
\begin{tabular}
[c]{ll}%
$0$ & if $ij^{n}\in\left(  S\cap I_{i}\right)  \cup\left(  S\cap J_{i}\right)
\cup\left(  S^{c}\cap K_{i}\right)  ,$\\
$1$ & if $ij^{n}\in\left(  S^{c}\cap I_{i}\right)  \cup\left(  S\cap
K_{i}\right)  $,
\end{tabular}
\right.
\]
and
\[
\tau\left(  S\right)  =\sum_{i,j}\tau\left(  S;ij^{n}\right)  \text{.}%
\]
Then $\tau\left(  S\right)  $ counts the number of external edges of the
subSierpinski graph, $\left\{  i\right\}  \times S(n-1,m),$ $0\leq i<m,$ cut
by $S$. Also let%
\[
\rho\left(  S;ij^{n}\right)  =\left\{
\begin{tabular}
[c]{ll}%
$0$ & if $ij^{n}\in\left(  S\cap I_{i}\right)  \cup\left(  S\cap J_{i}\right)
\cup\left(  S^{c}\cap K_{i}\right)  ,$\\
$m-j$ & if $ij^{n}\in\left(  S^{c}\cap I_{i}\right)  $\\
$j$ & if $ij^{n}\in\left(  S\cap K_{i}\right)  $.
\end{tabular}
\right.
\]
and
\[
\rho\left(  S\right)  =\sum_{i,j}\rho\left(  S;ij^{n}\right)  \text{.}%
\]
If $\forall h$, $S\cap V_{\left\{  h\right\}  \times S_{s.t}(n,m)}=Lex_{\pi
}^{-1}\left\{  1,2,...,\ell_{h}\right\}  $ for some $\pi\in\mathfrak{S}_{m}$
(which can be achieved with any $S$ by one round of $h$-compression, $0\leq
h<m$), then the number of exterior vertices in $S$, $\ell_{h}^{\prime
}=\left\vert \left\{  hj^{n}\in S\cap V_{\left\{  h\right\}  \times
S_{s.t}(n,m)}\right\}  \right\vert $, is determined by $\ell_{h}$: Given the
base $m$ representation of $\ell_{h}$,
\[
\ell_{h}=\sum_{i=1}^{n-1}\ell_{h,i}m^{n-1-i},
\]
then
\[
\ell_{h}^{\prime}=1+\max\left\{  j:j^{n-1}\leq_{Lex}\left(  \ell_{h,1}%
,\ell_{h,2},...,\ell_{h,n-1}\right)  \right\}  \text{.}%
\]
$\left\vert \Theta_{s,t}\left(  S\right)  \right\vert \geq0$ is integer valued
and nonincreasing under cyclic compression operations and so must eventually
be constant (it cannot decrease forever). Also, the number of internal edges
cut by $S$ will be constant after one cycle of compressions. Therefore the
difference, $\tau\left(  S\right)  $, must also eventually be constant.

\begin{lemma}
If $\forall h$, $S\cap V_{\left\{  h\right\}  \times S_{s.t}(n,m)}=Lex_{\pi
}^{-1}\left\{  1,2,...,\ell_{h}\right\}  $ for some $\pi\in\mathfrak{S}_{m}$,
then $\rho\left(  Comp_{Lex_{h}}\left(  S\right)  \right)  \leq\rho\left(
S\right)  $ with equality iff $Comp_{Lex_{h}}\left(  S\right)  =S$.

\begin{proof}
$S^{\prime}=Comp_{Lex_{h}}\left(  S\right)  $ alters $S$ in $\left\{
h\right\}  \times S_{s.t}(n,m)$ so as to minimize $\left\vert \Theta\left(
S^{\prime}\right)  \right\vert $ over all such alterations while maintaining
$\left\vert S^{\prime}\cap V_{\left\{  h\right\}  \times S_{s.t}%
(n,m)}\right\vert =\ell_{h}$. It is also the \textbf{unique} minimizer of
$\rho\left(  S^{\prime}\right)  $ over all such alterations. The sum for
$i\neq h$ (in the definition of $\rho\left(  S\right)  $) remains unchanged.
\end{proof}
\end{lemma}

\begin{theorem}
Cyclic compositions of $Comp_{Lex_{h(\operatorname{mod}m)}}\left(  S\right)
$, $h=0,1,...$, will eventually be constant, defining a nonmonotone Steiner
operation, $Comp_{\infty}$, on $S_{s,t}(n+1,m)$.

\begin{proof}
The cyclic compositions, $Comp_{Lex_{0}}\left(  S\right)  ,Comp_{Lex_{1}%
}\left(  Comp_{Lex_{0}}\left(  S\right)  \right)  ,...$ have nonincreasing
values of $\rho$. If those values do not decrease through a full cycle of $m$
consecutive applications of $Comp_{Lex_{h}}$, then by Lemma 1
$Comp_{Lex_{h(\operatorname{mod}m)}}\left(  S^{\prime}\right)  =S^{\prime}$,
$Comp_{Lex_{h+1(\operatorname{mod}m)}}\left(  S^{\prime}\right)  =S^{\prime}%
$,..,$Comp_{Lex_{h+m-1(\operatorname{mod}m)}}\left(  S^{\prime}\right)  =$
$S^{\prime}$ and we are then repeating the same compressions that already left
$S^{\prime}$ unchanged. In that case $Comp_{\infty}\left(  S\right)
=S^{\prime}$. This must happen after some finite number of compositions
because $\rho\left(  S\right)  $ is a (finite) nonnegative integer and each
change will decrease that integer by at least $1$. Since $\rho\left(
S^{\prime}\right)  $ must remain a nonnegative integer, the number of rounds
of cyclic composition required is at most $1+m\rho\left(  S\right)  $.
\end{proof}
\end{theorem}

\begin{corollary}
If $Comp_{\infty}\left(  S\right)  =S^{\prime}$, then $S^{\prime}$ is
$h$-compressed for every $h=0,1,...,m-1$.
\end{corollary}

\subsection{Stabilization Redux}

In extending the solution of the $EIP$ from $K_{m}^{n}$ to $S(n,m)$,
compression is the most essential of the three basic Steiner operations. To
make compression work we had to modify the concept of a graph. We added sets
of "external" vertices, $V_{I}=\left\{  v_{i}:i\in I\right\}  $,
$V_{K}=\left\{  v_{i}:i\in K\right\}  $, to $S(n,m)$ so that $S_{s,t}(n,m)$ is
a generalization of the subSierpinski graph, $\left\{  h\right\}  \times
S(n,m)$. So, can stabilization also be extended to $S_{s,t}(n,m)$? In terms of
the definition of stabilization in \cite{Har04}, the symmetry group of
$S_{s,t}(n,m)$ is only $\mathfrak{S}_{s}\times\mathfrak{S}_{t}\times
\mathfrak{S}_{m-s-t}$, so the number of stabilizing reflections is much less
($\binom{s}{2}+\binom{t}{2}+\binom{m-s-t}{2}$) than for $S(n,m)$ (where it is
$\binom{m}{2}$). However, every stabilization operation on $S(n,m)$ still acts
on the subsets of vertices of $S_{s,t}(n,m)$. Does that action induce a
Steiner operation?

\begin{theorem}
Any stabilization, $Stab_{i,j}:2^{V_{S(n,m)}}\rightarrow2^{V_{S(n,m)}}$, is
still a Steiner operation on $S_{s,t}(n,m)$.

\begin{proof}
Properties 1) \& 3) are not effected by the "exterior" vertices. Property 2)
could be destroyed but is not because, whatever $S$ is, if $i<j$ replacing
$j^{n}$ in $S$ by $i^{n}$ in $Stab_{i,j}\left(  S\right)  $ cannot increase
$\left\vert \Theta_{s,t}\left(  S\right)  \right\vert $: Either ($i,j\in I$ or
$J$ or $K$) or ($i\in I$ \& $j\in$ $J$ ) or ($i\in$ $I$ \& $j\in$ $K$) or
($i\in$ $J$ \& $j\in$ $K$). Whichever of the 6 possibilities is manifested,
$Stab_{i,j}\left(  S\right)  $, will not cut more exterior edges than $S$ (and
will cut fewer in the last 3 cases listed).
\end{proof}
\end{theorem}

\begin{corollary}%
\[
Stab_{i,j}(Lex_{\pi}^{-1}\left(  \left\{  1,2,...,\ell\right\}  \right)
)=\left\{
\begin{array}
[c]{ll}%
Lex_{\pi}^{-1}\left(  \left\{  1,2,...,\ell\right\}  \right)  & \text{if }%
\pi\left(  i\right)  <\pi\left(  j\right)  ,\\
Lex_{\pi\circ\left(  ij\right)  }^{-1}\left(  \left\{  1,2,...,\ell\right\}
\right)  & \text{if }\pi\left(  i\right)  >\pi\left(  j\right)  .
\end{array}
\right.
\]

\begin{proof}
This is the extension of 4.3.2(3) to $S_{s,t}(n,m)$.
\end{proof}
\end{corollary}

\begin{corollary}
$Lex^{-1}\left(  \left\{  1,2,...,\ell\right\}  \right)  $ is stable (wrt all
$i<j$).
\end{corollary}

\begin{corollary}
The stabilization-order of $S_{s,t}(n,m)$ is the same as that of $S(n,m)$.
\end{corollary}

After compressing $S\subseteq V_{S_{s,t}(n+1,m)}$ in every subSierpinski
graph, $\left\{  h\right\}  \times S_{s.t}(n,m)$ with $0\leq h<m$, $S$ need
not be stable. If we then apply the (extended) stabilization, $Stab_{i,j}$,
will $Stab_{i,j}(S)$ still be compressed?

\begin{lemma}
If $S\subseteq V_{S_{s,t}(n+1,m)}$ is $h$-compressed (\textit{i.e.
}$Comp_{Lex_{h}}\left(  S\right)  =S$), then either

\begin{enumerate}
\item $\boldsymbol{\ell}\left(  Stab_{i,j}(S)\right)  >\boldsymbol{\ell
}\left(  S\right)  $ or

\item $Stab_{i,j}(S)$ is $h$-compressed.
\end{enumerate}

\begin{proof}
In calculating $Comp_{Lex_{h}}\left(  S\right)  $ we first identify
\begin{align*}
I_{h}  &  =\left\{  i\in\left\{  0,1,...,m-1\right\}  :\left(  \left(  i\neq
h\right)  \&\left(  ih^{n}\in S\right)  \right)  \text{ or}\ \left(  \left(
i=h\right)  \&\left(  h\in I\right)  \right)  \right\}  ,\\
K_{h}  &  =\left\{  i\in\left\{  0,1,...,m-1\right\}  :\left(  i\neq h\right)
\&\left(  ih^{n}\notin S\right)  \text{ or }\left(  \left(  i=h\right)
\&\left(  h\in K\right)  \right)  \right\}  ,\\
J_{h}  &  =\left\{  0,1,...,m-1\right\}  -I_{h}-K_{h}\text{.}%
\end{align*}
(Note that $J_{h}=\emptyset$ unless $h\in J$ which means that $J_{h}=\left\{
h\right\}  )$. So $I_{h}J_{h}K_{h}=\pi\in\mathfrak{S}_{m}$ defines $Lex_{h}$,
the lexicographic order on $V_{S_{s,t}(n+1,m)}$ with components in the order
$I_{h}J_{h}K_{h}$. Now $\boldsymbol{\ell}\left(  Stab_{i,j}(S)\right)
\geq\boldsymbol{\ell}\left(  S\right)  $, so either $\boldsymbol{\ell}\left(
Stab_{i,j}(S)\right)  >\boldsymbol{\ell}\left(  S\right)  $ or
$\boldsymbol{\ell}\left(  Stab_{i,j}(S)\right)  =\boldsymbol{\ell}\left(
S\right)  $. The latter implies that $\ell_{h}\left(  Stab_{i,j}(S)\right)
=\ell_{h}\left(  S\right)  $ for $0\leq h<m$ which means that $Stab_{i,j}$
maps $S\cap\left(  \left\{  h\right\}  \times S_{s,t}(n,m)\right)  $ into
$\left\{  h\right\}  \times S_{s,t}(n,m)$. Since $S$ is $h$-compressed,
$S\cap\left(  \left\{  h\right\}  \times S_{s,t}(n,m)\right)  =Lex_{\pi}%
^{-1}\left(  \left\{  1,2,...,\ell_{h}\right\}  \right)  $ and by Corollary 2
\[
Stab_{i,j}(S)\cap\left(  \left\{  h\right\}  \times S_{s,t}(n,m)\right)
=\left\{
\begin{array}
[c]{ll}%
Lex_{\pi}^{-1}\left(  \left\{  1,2,...,\ell_{h}\right\}  \right)  & \text{if
}\pi\left(  i\right)  <\pi\left(  j\right) \\
Lex_{\pi\circ\left(  ij\right)  }^{-1}\left(  \left\{  1,2,...,\ell
_{h}\right\}  \right)  & \text{if }\pi\left(  i\right)  >\pi\left(  j\right)
\end{array}
\right.  \text{.}%
\]
In either case, $Stab_{i,j}(S)$ is $h$-compressed.
\end{proof}
\end{lemma}

\begin{theorem}
If $S\subseteq V_{S_{s,t}(n+1,m)}$ is compressed ($h$-compressed for
$h=0,1,...,m-1$), then either

\begin{enumerate}
\item $\boldsymbol{\ell}\left(  Stab_{\infty}(S)\right)  >\boldsymbol{\ell
}\left(  S\right)  $ or

\item $Stab_{\infty}(S)$ is compressed.
\end{enumerate}

\begin{proof}
Apply stabilizations $Stab_{i,j}$ cyclically. Eventually the composition are
constant, defining $Stab_{\infty}$. Since stabilization does not decrease
$\boldsymbol{\ell}\left(  S\right)  $, either some $Stab_{i,j}$ increases it
or it remains constant ($=\boldsymbol{\ell}\left(  S\right)  $) until
$Stab_{\infty}(S)$ is reached. In the latter case, Lemma 2 says that
$Stab_{\infty}(S)$ is $h$-compressed for $h=0,1,...,m-1$.
\end{proof}
\end{theorem}

\subsection{Subadditivity}

\begin{definition}
A function, $f:\mathbb{Z}_{N}\rightarrow\mathbb{Z}$, is called
\textit{subadditive }$(\operatorname{mod}N)$ if

\begin{enumerate}
\item $f(0)=0$,

\item $\forall x,y\in\mathbb{Z}_{N}$, $f(\left(  x+y\right)  \left(
\operatorname{mod}N\right)  )\leq f(x)+f(y)$.
\end{enumerate}
\end{definition}

\begin{example}
Bernstein's Lemma (Lemma 1.3 of \cite{Har04}) is equivalent to \linebreak%
$\left\vert \Theta\left(  Lex^{-1}\left\{  1,2,...,\ell\right\}  \right)
\right\vert $ on $Q_{n}$ being subadditive, \textit{i.e.} that
\end{example}

\begin{enumerate}
\item $\left\vert \Theta\right\vert \left(  Q_{n};0\right)  =0$,

\item $\forall k,\ell\in\mathbb{Z}_{2^{n}}$, $\left\vert \Theta\right\vert
\left(  Q_{n};\left(  k+\ell\right)  \operatorname{mod}2^{n}\right)
\leq\left\vert \Theta\right\vert \left(  Q_{n};k\right)  +\left\vert
\Theta\right\vert \left(  Q_{n};\ell\right)  $.
\end{enumerate}

Note that if $x$ is $0$ then $f(x+y)=f(0+y)=f(y)=0+f(y)=$ $f(x)+f(y)$ and
similarly if $y=0$.

A subadditive function, $f:\mathbb{Z}_{N}\rightarrow\mathbb{Z}$, is called
\textit{strongly subadditive} if
\[
\forall x,y\neq0\text{, }f(\left(  x+y\right)  \left(  \operatorname{mod}%
N\right)  )<f(x)+f(y)\text{.}%
\]

From this point on, due to the complexity of our proof technique for $m>3$, we
must restrict some theorems to $m=3$. Our goal is to prove that
\[
\left\vert \Theta\left(  Lex^{-1}\left(  \left\{  1,2,...,\ell\right\}
\right)  \right)  \right\vert :\mathbb{Z}_{m^{n}}\rightarrow\mathbb{Z}_{m^{n}}%
\]
is the edge-isoperimetric profile of $S(m,n)$. We cannot yet assert that
\[
\left\vert \Theta\right\vert \left(  S(m,n);\ell\right)  =\left\vert
\Theta\left(  Lex^{-1}\left(  \left\{  1,2,...,\ell\right\}  \right)  \right)
\right\vert
\]
so we need a notation for the latter that is more compact yet fully
descriptive. To this end we let
\[
\left\vert \Theta\left(  L^{-1}\right)  \right\vert \left(  n,m;\ell\right)
=\left\vert \Theta\left(  Lex^{-1}\left(  \left\{  1,2,...,\ell\right\}
\right)  \right)  \right\vert
\]

\[%
{\parbox[b]{4.5455in}{\begin{center}
\includegraphics[
trim=0.000000in 1.534181in 0.282943in 0.618515in,
height=4.8966in,
width=4.5455in
]%
{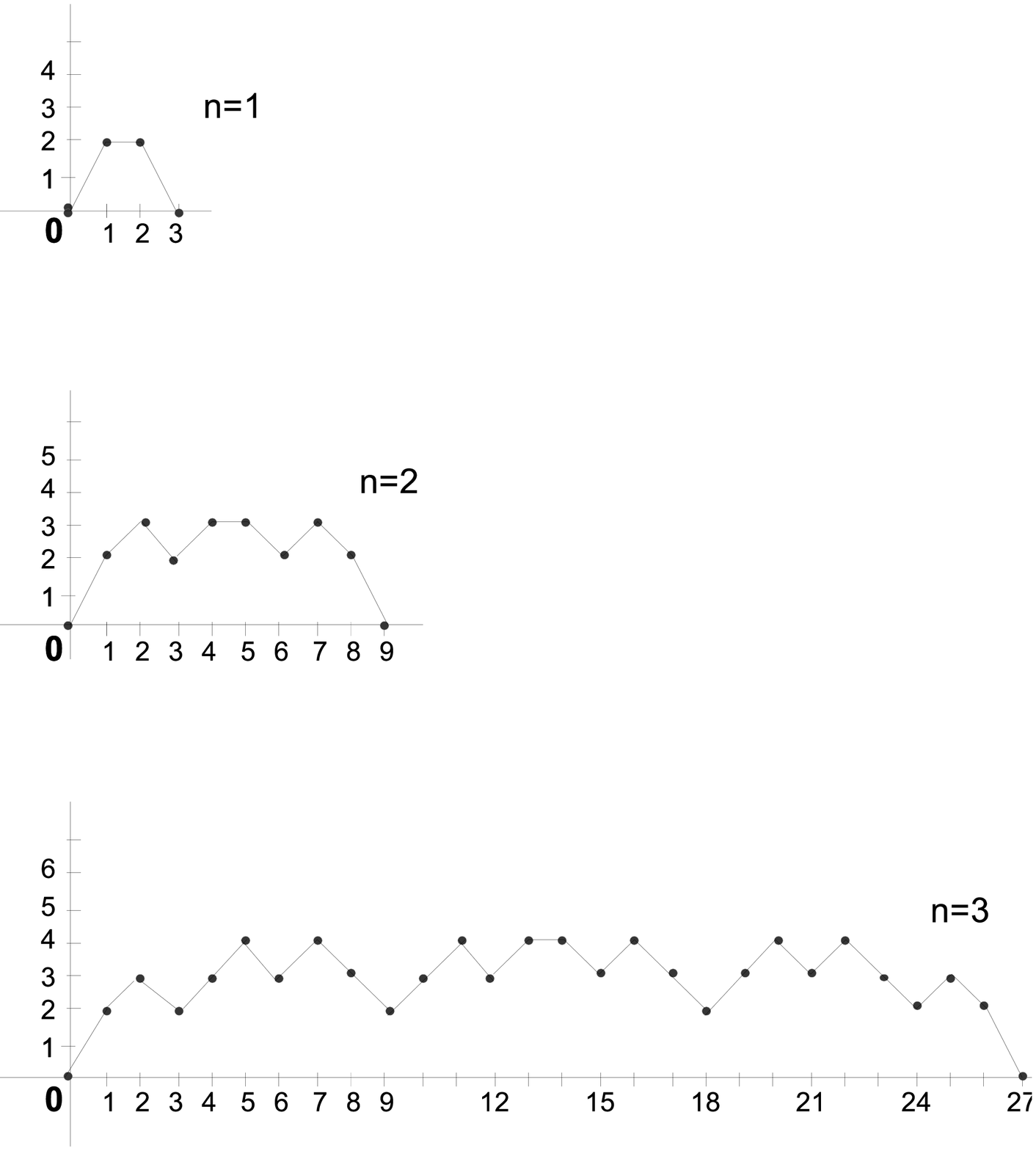}%
\\
Figure 7-Graphs of $\left\vert \Theta\left(  L^-1\right)  \right\vert \left(
n,3;\ell\right)  $, $n=1,2,3$.
\end{center}}}%
\]
The reader might compare these with the diagrams of $S(n,3)$, $n=1,2,3$. They
are in Figure 2 of \cite{Har15'}.

\begin{theorem}
$\forall n,\left\vert \Theta\left(  L^{-1}\right)  \right\vert \left(
n,3;\ell\right)  $ is strongly subadditive (as a function of $\ell
\in\mathbb{Z}_{3^{n}}$).
\end{theorem}

Before the proof of Theorem 7, we need some technical lemmas. Let%

\[
\left\vert \Theta_{0}\right\vert (n,3;\ell\boldsymbol{)=}\left\{
\begin{tabular}
[c]{ll}%
$0$ & $\text{if }\ell=0,$\\
$1$ & if \ $0<\ell<3^{n-1}/2$ or $3^{n}-3^{n-1}/2<\ell<3^{n},$\\
$2$ & if $3^{n-1}/2<\ell<3^{n}-3^{n-1}/2$.
\end{tabular}
\right.
\]
$\ $and$\ $%
\[
\left\vert \Theta_{1}\right\vert (n,3;\ell\boldsymbol{)=}\left\{
\begin{array}
[c]{ll}%
0 & \text{if }\ell=0\text{ }\left(  \operatorname{mod}3^{n-1}\right)  ,\\
\left\vert \Theta\left(  L^{-1}\right)  \right\vert \left(  n-1,3;\ell\left(
\operatorname{mod}3^{n-1}\right)  \right)  -1 & \text{if\ }\ell\neq0\text{
}\left(  \operatorname{mod}3^{n-1}\right)  \text{.}%
\end{array}
\right.
\]

\begin{lemma}
$\left\vert \Theta_{0}\right\vert (n,3;\ell\boldsymbol{)+}\left\vert
\Theta_{1}\right\vert (n,3;\ell\boldsymbol{)=}\left\vert \Theta\left(
L^{-1}\right)  \right\vert \left(  n,3;\ell\right)  $.

\begin{proof}
$\left\vert \Theta_{1}\right\vert (n,3;\ell\boldsymbol{)}$ is essentially the
contribution of the edges interior to $\left\{  i\right\}  \times S\left(
n-1,3\right)  $, $i=0,1,2$, to $\left\vert \Theta\left(  L^{-1}\right)
\right\vert \left(  n,3;\ell\right)  $. We say "essentially" because when
$10^{n-1}$ is numbered, $Lex\left(  10^{n-1}\right)  =3^{n-1}+1$, $\left\vert
\Theta\right\vert $ increases by $1$ from its predecessor $02^{n-1}$ (and
similarly with $20^{n-1}$). However at $0^{n}$ the increase is from $0$ to $2
$, but this is compensated for by adding $\left\vert \Theta_{0}\right\vert
(n,3;1\boldsymbol{)}=1$ to $\left\vert \Theta_{1}\right\vert (n,3;1)=1$. At
$\ell=\left(  3^{n-1}+1\right)  /2$ there is another external edge, $\left\{
01^{n-1},10^{n-1}\right\}  $, added to the edge-boundary of $Lex^{-1}\left(
\left\{  1,2,...,\ell\right\}  \right)  $, so $\left\vert \Theta
_{0}\right\vert $ increases from $1$ to $2$ at that point. The decreases on
the other side of $3^{n}/2$ are symmetric.
\end{proof}
\end{lemma}

Lemma 2 gives a recursive procedure for computing the isoperimetric profile
$\left\vert \Theta\right\vert (S(n,3);\ell\boldsymbol{)}$. There is also a
formula that holds for all $m$: Given the base $m$ expansion of $\ell-1$,%
\[
\ell-1=\sum_{i=1}^{n}\ell_{m,i}m^{n-i}\text{,}%
\]
let%
\[
\ell_{m,i}^{\prime}=1+\max\left\{  j:j^{n-i}\leq_{Lex}\left(  \ell
_{m,i+1},\ell_{m,i+2},...,\ell_{m,n}\right)  \right\}
\]
then

\begin{proposition}
$\left\vert \Theta\left(  L^{-1}\right)  \right\vert \left(  n,m;\ell\right)
\boldsymbol{=}\sum_{h=1}^{n}\ell_{m,h}\left(  m-\ell_{m,h}\right)  +\left\vert
\ell_{m,h}^{\prime}-\ell_{m,h}\right\vert -\ell_{m,h}$.

\begin{proof}
Note that%
\[
\left\vert \ell_{m,h}^{\prime}-\ell_{m,h}\right\vert -\ell_{m,h}=\left\{
\begin{array}
[c]{ll}%
-\ell_{m,h}^{\prime} & \text{if }\ell_{m,h}^{\prime}\leq\ell_{m,h}\\
\ell_{m,h}^{\prime}-2\ell_{m,h} & \text{if }\ell_{m,h}^{\prime}\geq\ell_{m,h}%
\end{array}
\text{.}\right.
\]
Lex numbering proceeds through the Sierpinski subgraphs of level $h$,
$v^{\left(  h\right)  }\times S(n-h,m)$, $1\leq h\leq n$, in Lex order (wrt
$v^{\left(  h\right)  }$), completely numbering each one before moving on to
the next. The definition of $S(n,m)$ says that the edges at level $h$ connect
a corner of some $v^{\left(  h\right)  }\times S(n-h,m)$ to a corner of
$u^{\left(  h\right)  }\times S(n-h,m)$ where $v_{i}=u_{i}$, for $1\leq i<h$,
but $v_{h}=j\neq k=u_{h}$. The edge then connects $v^{\left(  h\right)
}k^{n-h}$ to $u^{\left(  h\right)  }j^{n-h}$. Since $\ell_{m,h}$ is the number
of copies of $S(n-h,m)$ contained in some $S(n-h+1,m)$ that are completely
contained in $Lex^{-1}\left(  \left\{  1,2,...,\ell\right\}  \right)  $,
$\ell_{m,h}\left(  m-\ell_{m,h}\right)  $ is the number of those edges cut by
$Lex^{-1}\left(  \left\{  1,2,...,\ell\right\}  \right)  $. The next copy of
$S(n-h,m)$ is only partially contained in $Lex^{-1}\left(  \left\{
1,2,...,\ell\right\}  \right)  $ and $\ell_{m,h}^{\prime}$ is the number of
its corner vertices that are in $Lex^{-1}\left(  \left\{  1,2,...,\ell
\right\}  \right)  $. If $\ell_{m,h}^{\prime}<\ell_{m,h}$, each of those
corner vertices cover up the other end of an edge that had been counted by
$\ell_{m,h}\left(  m-\ell_{m,h}\right)  $, so its marginal contribution to the
edge-boundary is $-1$. If $\ell_{m,h}^{\prime}>\ell_{m,h}$, those corner
vertices are cutting new edges so its marginal contribution is $+1$. Since all
edges are at some level, $h$, and the sets of edges involved at different
levels are disjoint, the contributions of each level add up so $\left\vert
\Theta\left(  L^{-1}\right)  \right\vert \left(  n,3;\ell\right)
\boldsymbol{=}\sum_{h=1}^{n}\ell_{m,h}\left(  m-\ell_{m,h}\right)  +\left\vert
\ell_{m,h}^{\prime}-\ell_{m,h}\right\vert -\ell_{m,h}$.
\end{proof}
\end{proposition}

\begin{lemma}
If $f,g:\mathbb{Z}_{N}\rightarrow\mathbb{Z}$ are subadditive and
$a,b\in\mathbb{Z}^{+}$, then $af+bg$ is subadditive.

\begin{proof}%
\begin{align*}
\left(  af+bg\right)  \left(  x+y\right)   &  =af\left(  x+y\right)
+bg\left(  x+y\right) \\
&  \leq\left(  af\left(  x\right)  +af\left(  y\right)  \right)  +\left(
bg\left(  x\right)  +bg\left(  y\right)  \right) \\
&  =\left(  af+bg\right)  \left(  x\right)  +\left(  af+bg\right)  \left(
y\right)  \text{.}%
\end{align*}

\end{proof}
\end{lemma}

\begin{lemma}
If $j<n$ and $3^{j-1}/2<\ell<3^{j}/2$, then $\left\vert \Theta_{1}\right\vert
(n,3;\ell\boldsymbol{)=}\left\vert \Theta_{1}\right\vert (j,3;\ell
\boldsymbol{)}+1$.

\begin{proof}
If $3^{j-1}/2<\ell<3^{j}/2$ then
\begin{align*}
\left\vert \Theta_{1}\right\vert (j+1,3;\ell\boldsymbol{)}  &  \boldsymbol{=}%
\left\vert \Theta\left(  L^{-1}\right)  \right\vert \left(  j,3;\ell\left(
\operatorname{mod}3^{j}\right)  \right)  -1\text{ by the definition of
}\left\vert \Theta_{1}\right\vert \text{,}\\
&  =\left\vert \Theta_{0}\right\vert (j,3;\ell\boldsymbol{)+}\left\vert
\Theta_{1}\right\vert (j,3;\ell)-1\text{ by Lemma 2,}\\
&  =2\boldsymbol{+}\left\vert \Theta_{1}\right\vert (j,3;\ell)-1\text{ by the
definition of }\left\vert \Theta_{0}\right\vert \text{,}\\
&  =\left\vert \Theta_{1}\right\vert (j,3;\ell)+1\text{.}%
\end{align*}
Similarly, for $n>j+1$,
\begin{align*}
\left\vert \Theta_{1}\right\vert (n,3;\ell\boldsymbol{)}  &  \boldsymbol{=}%
\left\vert \Theta\left(  L^{-1}\right)  \right\vert \left(  n-1,3;\ell\left(
\operatorname{mod}3^{n-1}\right)  \right)  -1\\
&  =\left\vert \Theta_{0}\right\vert (n-1,3;\ell\boldsymbol{)+}\left\vert
\Theta_{1}\right\vert (n-1,3;)-1\\
&  =1\boldsymbol{+}\left\vert \Theta_{1}\right\vert (n-1,3;\ell)-1\\
&  =\left\vert \Theta_{1}\right\vert (n-1,3;\ell)\\
&  =...\\
&  =\left\vert \Theta_{1}\right\vert (j+1,3;\ell\boldsymbol{)}\\
&  =\left\vert \Theta_{1}\right\vert (j,3;\ell)+1\text{.}%
\end{align*}

\end{proof}
\end{lemma}

\begin{lemma}
For $j=0,1,...,n$, if $3^{n}/2-3^{j}/2\leq\ell<3^{n}/2$ then
\begin{align*}
\left\vert \Theta_{1}\right\vert (n,3;\ell\boldsymbol{)}  &  \boldsymbol{=}%
\left\vert \Theta_{1}\right\vert (j,3;\ell-\left(  3^{n}/2-3^{j}/2\right)
\boldsymbol{)+}\left(  n-j\right) \\
&  =\left\vert \Theta_{1}\right\vert (j,3;\ell\operatorname{mod}%
3^{j}\boldsymbol{)+}\left(  n-j\right)  \text{.}%
\end{align*}

\begin{proof}
If $j=n$ the result says that $\left\vert \Theta_{1}\right\vert (n,3;\ell
\boldsymbol{)}\boldsymbol{=}\left\vert \Theta_{1}\right\vert (n,3;\ell-0)+0$
which is true. If $j<n$ then we proceed by induction on $n$. For $n=1$, the
only value of $j$ we need to consider is $j=0$. In that case, $3^{n}%
/2-3^{j}/2=3^{1}/2-3^{0}/2=1$ and $1\leq\ell<3^{1}/2$ implies $\ell=1$. Then
$\left\vert \Theta_{1}\right\vert (n,3;\ell\boldsymbol{)=}\left\vert
\Theta_{1}\right\vert (1,3;1)=1$ and $\left\vert \Theta_{1}\right\vert
(j,3;\ell-\left(  3^{n}/2-3^{j}/2\right)  \boldsymbol{)+}\left(  n-j\right)
=\left\vert \Theta_{1}\right\vert (0,3;1-\left(  3^{1}/2-3^{0}/2\right)
\boldsymbol{)+}\left(  1-0\right)  =\left\vert \Theta_{1}\right\vert
(0,3;0\boldsymbol{)+}1=1$. So we have established the base case. Assume the
theorem is true for some $n-1\geq1$. We may assume that $j<n$, so $\left(
3^{n}/2-3^{j}/2\right)  -3^{n-1}=3^{n-1}/2-3^{j}/2\leq\ell-3^{n-1}<3^{n}/2$.
Now from the definition of $\left\vert \Theta_{1}\right\vert $, $\left\vert
\Theta_{1}\right\vert (n,3;\ell\boldsymbol{)=}\left\vert \Theta_{1}\right\vert
(n-1,3;\ell-3^{n-1}\boldsymbol{)+}1$, and by the inductive hypothesis,
$\left\vert \Theta_{1}\right\vert (n-1,3;\ell-3^{n-1}\boldsymbol{)=}\left\vert
\Theta_{1}\right\vert (j,3;\ell-3^{n-1}-\left(  3^{n-1}/2-3^{j}/2\right)
\boldsymbol{)+}\left(  n-1\right)  -j$. Therefore $\left\vert \Theta
_{1}\right\vert (n,3;\ell\boldsymbol{)=}\left\vert \Theta_{1}\right\vert
(j,3;\ell-\left(  3^{n}/2-3^{j}/2\right)  \boldsymbol{)+}\left(  n-j\right)  $.
\end{proof}
\end{lemma}

\begin{lemma}
$\left\vert \Theta_{0}\right\vert (n,3;\ell\boldsymbol{)}$ is subadditive.

\begin{proof}
By direct computation from the definition of $\left\vert \Theta_{0}\right\vert
$.
\end{proof}
\end{lemma}

\begin{definition}
A function $g:[mN]\rightarrow\mathbb{Z}$ is the $m$-replicate of
$f:[N]\rightarrow\mathbb{Z}$ if $g\left(  \ell\right)  =f\left(
\ell(\operatorname{mod}N)\right)  $.
\end{definition}

\begin{example}
$\left\vert \Theta_{1}\right\vert (n,3;\ell\boldsymbol{)}$ is the
$3$-replicate of $\left\vert \Theta\left(  L^{-1}\right)  \right\vert \left(
n-1,3;\ell(\operatorname{mod}3^{n-1})\right)  -1$ except that the value of
$-1$ at $\ell=0,3^{n-1}$ is replaced by $0$.
\end{example}

\begin{lemma}
If $f:\mathbb{Z}_{N}\rightarrow\mathbb{Z}$ is subadditive, then its
$m$-replicate, $g:\mathbb{Z}_{mN}\rightarrow\mathbb{Z}$, is also subadditive.

\begin{proof}%
\begin{align*}
g(\left(  k+\ell\right)  \operatorname{mod}mN)  &  =f(\left(  k+\ell\right)
(\operatorname{mod}N))\\
&  \leq f(\left(  k\right)  (\operatorname{mod}N))+f(\left(  \ell\right)
(\operatorname{mod}N))\\
&  =g(k)+g(\ell)\text{.}%
\end{align*}

\end{proof}
\end{lemma}

\begin{lemma}
If $\left\vert \Theta\left(  L^{-1}\right)  \right\vert \left(  n,3;\ell
\right)  $ is strongly subadditive, then $\left\vert \Theta_{1}\right\vert
(n+1,3;\ell\boldsymbol{)}$ is subadditive.
\end{lemma}

\begin{proof}
If $\left\vert \Theta\left(  L^{-1}\right)  \right\vert \left(  n,m;\ell
\right)  $ is strongly subadditive, then the function
\[
\left\vert \Theta_{1}^{\prime}\right\vert (n,3;\ell\boldsymbol{)=}\left\{
\begin{array}
[c]{ll}%
0 & \text{if }\ell=0\text{ }\\
\left\vert \Theta\left(  L^{-1}\right)  \right\vert \left(  n,m;\ell\right)
-1 & if\ \ell\neq0\text{ }%
\end{array}
\right.
\]
is subadditive. $\left\vert \Theta_{1}\right\vert (n+1,3;\ell\boldsymbol{)} $
is just the $3$-replicate of $\left\vert \Theta_{1}^{\prime}\right\vert
(n,3;\ell\boldsymbol{)}$, so$\ $by Lemma 7, $\left\vert \Theta_{1}\right\vert
(n+1,3;\ell\boldsymbol{)}$ is subadditive.
\end{proof}

\begin{proof}
(of theorem 7) It follows from Lemmas 3, 4, 7 \& 9 that if $\left\vert
\Theta\left(  L^{-1}\right)  \right\vert \left(  n,3;\ell\right)  $ is
strongly subadditive, then $\left\vert \Theta\left(  L^{-1}\right)
\right\vert \left(  n+1,3;\ell\right)  $ is subadditive. However, to complete
the induction step for theorem 7, we must show that $\left\vert \Theta\left(
L^{-1}\right)  \right\vert \left(  n+1,3;\ell\right)  $ is \textbf{strongly}
subadditive. For that we argue by contradiction: If $\left\vert \Theta\left(
L^{-1}\right)  \right\vert \left(  n+1,3;\ell\right)  $ is not strongly
subadditive, $\exists k,\ell>0$ such that $\Delta=$
\begin{align*}
&  \left\vert \Theta\left(  L^{-1}\right)  \right\vert \left(  n+1,3;k+\ell
\right)  \boldsymbol{-}\left(  \left\vert \Theta\left(  L^{-1}\right)
\right\vert \left(  n+1,3;k\right)  \boldsymbol{+}\left\vert \Theta\left(
L^{-1}\right)  \right\vert \left(  n+1,3;\ell\right)  \right) \\
&  =0\text{.}%
\end{align*}
However $\Delta=\Delta_{0}+\Delta_{1}$ and since $\left\vert \Theta
_{0}\right\vert (S(n+1,3);\ell\boldsymbol{)}$ and $\left\vert \Theta
_{1}\right\vert (S(n+1,3);\ell\boldsymbol{)}$ are both subadditive,
$\Delta_{0},\Delta_{1}\leq0$, so the only way that $\Delta$ can be $0$ is if
both $\Delta_{0},\Delta_{1}=0$.

The zeroes of $\Delta_{0}$ are easy to identify: If $0<k\leq\ell$ \&
$k<3^{n+1}/2$, only $\ell<3^{n}/2$ can give a zero and then $\left\vert
\Theta_{0}\right\vert (n+1,3;k\boldsymbol{)=}1\boldsymbol{=}\left\vert
\Theta_{0}\right\vert (n+1,3;\ell\boldsymbol{)}$ \& $\left\vert \Theta
_{0}\right\vert (n+1,3;k+\ell\boldsymbol{)=}2$. This happens iff $k,\ell
<3^{n}/2$ and $k+\ell>3^{n}/2$. Since $k\leq\ell$, $\ell>3^{n}/4$. For
$3^{n+1}/2<k,\ell<3^{n+1}$ there are symmetric conditions.

To show that none of these values of $k,\ell$ that give a zero of $\Delta_{0}$
also give a zero of $\Delta_{1}$, we partition the values of $k$ by
$k_{j}=\sum_{i=0}^{j}3^{i}=\left(  3^{j+1}-1\right)  /2$, $j=0,1,...,n-1$.
$k_{0}=1$, the minimum possible value of $k$. If $k=k_{0}=1$, then $\ell<$
$3^{n}/2$\ and $1+\ell>3^{n}/2$ has only one solution, $\ell=\left(
3^{n}-1\right)  /2$. But then $\left\vert \Theta_{1}\right\vert
(n+1,3;1\boldsymbol{)=}1$ whereas
\begin{align*}
&  \left\vert \Theta_{1}\right\vert (n+1,3;\left(  3^{n}+1\right)
/2\boldsymbol{)-}\left(  \left\vert \Theta_{1}\right\vert (n+1,3;\left(
3^{n}-1\right)  /2\boldsymbol{)}\right) \\
&  =\left\vert \Theta\left(  L^{-1}\right)  \right\vert \left(  n,3;\left(
3^{n}+1\right)  /2\boldsymbol{)}\right)  \boldsymbol{-}\left\vert
\Theta\left(  L^{-1}\right)  \right\vert \left(  n,3;\left(  3^{n}-1\right)
/2\boldsymbol{)}\right) \\
&  =(n+1)-(n+1)=0\text{,}%
\end{align*}
and we are done with $j=0$. If $j>0$ and $k_{j-1}<k\leq k_{j}$ then
$3^{n+1}/2-k_{j}=3^{n+1}/2-\left(  3^{j+1}-1\right)  /2$ so $3^{n+1}%
/2-3^{j+1}/2<\ell<3^{n+1}/2$. By Lemma 6 then $\left\vert \Theta
_{1}\right\vert (n+1,3;\ell\boldsymbol{)=}\left\vert \Theta_{1}\right\vert
(j+1,3;\ell-\left(  3^{n+1}/2-3^{j+1}/2\right)  \boldsymbol{)+}\left(  \left(
n+1\right)  -\left(  j+1\right)  \right)  $. Also, for $\left(  3^{j}%
-1\right)  /2=k_{j-1}<k\leq k_{j}=\left(  3^{j+1}-1\right)  /2<3^{j+1}/2$, by
Lemma 5 $\left\vert \Theta_{1}\right\vert (n+1,3;k\boldsymbol{)=}\left\vert
\Theta_{1}\right\vert (j+1,3;k\boldsymbol{)+}1$. Therefore if $k_{j-1}<k\leq
k_{j}$,%

\[%
\begin{tabular}
[c]{l}%
$\Delta_{1}=\left\vert \Theta_{1}\right\vert (n+1,3;k+\ell)-\left(  \left\vert
\Theta_{1}\right\vert (n+1,3;k)+\left\vert \Theta_{1}\right\vert
(n+1,3;\ell)\right)  $\\
$\ \ \ =\left\vert \Theta_{1}\right\vert j+1,3;\left(  k+\ell\right)  -\left(
3^{n+1}/2-3^{j+1}/2\boldsymbol{)}\right)  +\left(  \left(  n+1\right)
-\left(  j+1\right)  \right)  $\\
$\ \ \ \ \ \ \ -\left(  \left\vert \Theta_{1}\right\vert (j+1,3;k)+1\right)
$\\
$\ \ \ \ \ \ -\left(  \left\vert \Theta_{1}\right\vert (j+1,3;\ell-\left(
3^{n+1}/2-3^{j+1}/2\boldsymbol{)}\right)  +\left(  \left(  n+1\right)
-\left(  j+1\right)  \right)  \right)  $\\
$\ \ \ =-1+\left\vert \Theta_{1}\right\vert (j+1,3;\left(  k+\ell\right)
\operatorname{mod}3^{j+1})$\\
$\ \ \ \ \ \ -\left(  \left\vert \Theta_{1}\right\vert (j+1,3;k)+\left\vert
\Theta_{1}\right\vert (j+1,3;\ell(\operatorname{mod}3^{j+1})\right)  $\\
$\ \ \leq-1+0$ by the inductive hypothesis and Lemma 9.\\
$\ <0$.
\end{tabular}
\ \
\]

This covers all the values of $k$ (and the corresponding values of $\ell$).
\end{proof}

\begin{lemma}
If $0<k\leq\ell<3^{n}/2$ $\&$ $k+\ell>3^{n}/2$, then
\[
\left\vert \Theta\left(  L^{-1}\right)  \right\vert \left(  n,3;k\right)
\boldsymbol{+}\left\vert \Theta\left(  L^{-1}\right)  \right\vert \left(
n,3;\ell\right)  \boldsymbol{\ \geq}\left\vert \Theta\left(  L^{-1}\right)
\right\vert \left(  n,3;k+\ell\right)  +2.
\]

\begin{proof}
If $0<k<3^{n-1}/2$, then
\begin{align}
\left\vert \Theta\left(  L^{-1}\right)  \right\vert \left(  n,3;k\right)   &
\boldsymbol{=}\left\vert \Theta_{0}\right\vert (n,3;k\boldsymbol{)+}\left\vert
\Theta_{1}\right\vert (n,3;k\boldsymbol{)}\text{ by Lemma 3,}\nonumber\\
&  =1+\left(  \left\vert \Theta\left(  L^{-1}\right)  \right\vert \left(
n-1,3;k\right)  -1\right)  \text{ }\nonumber\\
&  \text{ \ \ \ \ \ \ \ by the definitions of }\Theta_{0}\text{ \& }\Theta
_{1}\text{,}\\
&  =\left\vert \Theta\left(  L^{-1}\right)  \right\vert \left(
n-1,3;k\right)  \text{.}\nonumber
\end{align}
Also since $\ell>3^{n}/2-3^{n-1}/2=3^{n-1}$, by the same reasoning,
\begin{align*}
\left\vert \Theta\left(  L^{-1}\right)  \right\vert \left(  n,3;\ell\right)
&  \boldsymbol{=}\left\vert \Theta_{0}\right\vert (n,3;\ell\boldsymbol{)+}%
\left\vert \Theta_{1}\right\vert (n,3;\ell\boldsymbol{)}\\
&  =2+\left(  \left\vert \Theta\left(  L^{-1}\right)  \right\vert \left(
n-1,3;\ell-3^{n-1}\right)  -1\right) \\
&  =\left\vert \Theta\left(  L^{-1}\right)  \right\vert \left(  n-1,3;\ell
-3^{n-1}\right)  +1\text{.}%
\end{align*}
Therefore%
\begin{align*}
&  \left\vert \Theta\left(  L^{-1}\right)  \right\vert \left(  n,3;k\right)
\boldsymbol{+}\left\vert \Theta\left(  L^{-1}\right)  \right\vert \left(
n,3;\ell\right) \\
&  =\left\vert \Theta\left(  L^{-1}\right)  \right\vert \left(
n-1,3;k\right)  +\left\vert \Theta\left(  L^{-1}\right)  \right\vert \left(
n-1,3;\ell\right)  +1\text{.}%
\end{align*}
Also, $k+\ell-3^{n-1}>3^{n}/2-3^{n-1}=3^{n-1}/2$, so by induction on $n$,
\begin{align*}
&  \left\vert \Theta\left(  L^{-1}\right)  \right\vert \left(  n-1,3;k\right)
\boldsymbol{+}\left\vert \Theta\left(  L^{-1}\right)  \right\vert \left(
n-1,3;\ell-3^{n-1}\right) \\
&  \geq\left\vert \Theta\left(  L^{-1}\right)  \right\vert \left(
n-1,3;k+\left(  \ell-3^{n-1}\right)  \right)  +2\text{,}%
\end{align*}
and we have
\begin{align*}
&  \left\vert \Theta\left(  L^{-1}\right)  \right\vert \left(  n,3;k\right)
\boldsymbol{+}\left\vert \Theta\left(  L^{-1}\right)  \right\vert \left(
n,3;\ell\right)  \left\vert \Theta\right\vert \\
&  \boldsymbol{\geq}\left(  \left\vert \Theta\left(  L^{-1}\right)
\right\vert \left(  n,3;k+\ell\right)  -1\right)  +1+2\\
&  =\left\vert \Theta\left(  L^{-1}\right)  \right\vert \left(  n,3;k+\ell
\right)  +2\text{.}%
\end{align*}
If, on the other hand, $k>3^{n-1}/2$, then $\ell>3^{n-1}/2$ and
\begin{align*}
\left\vert \Theta\left(  L^{-1}\right)  \right\vert \left(  n,3;k\right)   &
=\left\vert \Theta\left(  L^{-1}\right)  \right\vert \left(  n-1,3;k\left(
\operatorname{mod}3^{n-1}\right)  \right)  +1\text{,}\\
\left\vert \Theta\left(  L^{-1}\right)  \right\vert \left(  n,3;\ell\right)
&  =\left\vert \Theta\left(  L^{-1}\right)  \right\vert \left(  n-1,3;\ell
\left(  \operatorname{mod}3^{n-1}\right)  \right)  +1\text{,}\\
\left\vert \Theta\left(  L^{-1}\right)  \right\vert \left(  n,3;k+\ell\right)
&  \leq\left\vert \Theta\left(  L^{-1}\right)  \right\vert \left(
n-1,3;\left(  k+\ell\right)  \left(  \operatorname{mod}3^{n-1}\right)
\right)  +1\text{.}%
\end{align*}
Therefore%
\begin{align*}
&  \left\vert \Theta\left(  L^{-1}\right)  \right\vert \left(  n,3;k\right)
\boldsymbol{+}\left\vert \Theta\left(  L^{-1}\right)  \right\vert \left(
n,3;\ell\right) \\
&  =(\left\vert \Theta\left(  L^{-1}\right)  \right\vert \left(
n-1,3;\ell\left(  \operatorname{mod}3^{n-1}\right)  \right)  +1)\\
&  \text{ \ \ \ }+\left(  \left\vert \Theta\left(  L^{-1}\right)  \right\vert
\left(  n-1,3;\ell\left(  \operatorname{mod}3^{n-1}\right)  \right)
+1\right)  \text{,}\\
&  \geq(\left\vert \Theta\left(  L^{-1}\right)  \right\vert \left(
n-1,3;\left(  k+\ell\right)  \left(  \operatorname{mod}3^{n-1}\right)  \right)
\\
&  \text{ \ \ \ }+1)+2\text{, by strong subadditivity (Theorem 7),}\\
&  =\left(  \left\vert \Theta\left(  L^{-1}\right)  \right\vert \left(
n,3;k+\ell\right)  -1\right)  +3\text{,}\\
&  =\left\vert \Theta\left(  L^{-1}\right)  \right\vert \left(  n,3;k+\ell
\right)  +2\text{.}%
\end{align*}

\end{proof}
\end{lemma}

\begin{corollary}
\bigskip If $3^{n}/2<k\leq\ell<3^{n}$ $\&$ $k+\ell<3^{n+1}/2$, then
\[
\left\vert \Theta\left(  L^{-1}\right)  \right\vert \left(  n,3;k\right)
\boldsymbol{+}\left\vert \Theta\left(  L^{-1}\right)  \right\vert \left(
n,3;\ell\right)  \geq\left\vert \Theta\left(  L^{-1}\right)  \right\vert
\left(  n,3;\left(  k+\ell\right)  -3^{n}\right)  +2.
\]

\begin{proof}%
\begin{align*}
&  \left\vert \Theta\left(  L^{-1}\right)  \right\vert \left(  n,3;k\right)
\boldsymbol{+}\left\vert \Theta\left(  L^{-1}\right)  \right\vert \left(
n,3;\ell\right)  \boldsymbol{\ }\\
&  =\left\vert \Theta\left(  L^{-1}\right)  \right\vert \left(  n,3;3^{n}%
-k\right)  \boldsymbol{+}\left\vert \Theta\left(  L^{-1}\right)  \right\vert
\left(  n,3;3^{n}-\ell\right) \\
&  \text{ \ \ \ \ \ by the duality of }\Theta\text{,}\\
&  \boldsymbol{\geq}\left\vert \Theta\left(  L^{-1}\right)  \right\vert
\left(  n,3;2\cdot3^{n}-\left(  k+\ell\right)  \right)  +2\\
&  \text{ \ \ \ \ \ by Lemma 10,}\\
&  =\left\vert \Theta\left(  L^{-1}\right)  \right\vert \left(  n,3;\left(
k+\ell\right)  -3^{n}\right)  +2\\
&  \text{\ \ \ \ \ \ by the duality of }\Theta\text{ again.}%
\end{align*}

\end{proof}
\end{corollary}

\section{The Main Theorem}

\begin{theorem}
$S(n,3)$ has Lex nested solutions for $EIP$

\begin{proof}
(We actually prove Conjecture 2 for $m=3$, which implies Conjecture 1 for
$m=3$) By induction on $n$:

\begin{description}
\item[Initial Step] It is true for $n=1$, since in Corollary 1 (of Section
4.4) we noted that the stabilization-order of $S_{s,t}\left(  1,m\right)  $ is
the same as that of $S\left(  1,m\right)  $ which is the natural total order
$0<1<...<m-1$.

\item[Inductive Step] Assume the theorem is true for $n\geq1$ and that
\[
S\subseteq V_{S_{s,t}(n+1,m)}=\left\{  0,1,...,m-1\right\}  ^{n+1}%
\]
with $\left\vert S\right\vert =\ell$. We shall use the following three Steiner
operations to reduce any such $S$ to $Lex_{IJK}^{-1}\left(  \left\{
1,2,...,\ell\right\}  \right)  $ (since $IJK=\iota$, the identity permutation,
$Lex_{IJK}=Lex$):

\begin{enumerate}
\item Apply stabilization, so we need only consider $S$ that are stable,

\item Apply compression, utilizing the inductive hypothesis. Then we need only
consider $S$ that are compressed and stable,

\item Apply subadditivation (a StOp based on the subadditivity of
\linebreak$\left\vert \Theta\right\vert \left(  S(n,m);\ell\right)  $ reducing
$S$ to $Lex^{-1}(\left\{  1,...,\ell\right\}  )$.
\end{enumerate}
\end{description}

\underline{StOp $\mathbf{1}$-Stabilization}: See Section 4.4. $S$ is stable
iff it is an ideal in \linebreak$\mathcal{S}$-$\mathcal{O}\left(
S_{s,t}(n+1,m)\right)  =\mathcal{S}$-$\mathcal{O}\left(  S(n+1,m)\right)  $.
The structure of $\mathcal{S}$-$\mathcal{O}\left(  S(n+1,m)\right)  $ was
discussed in Sections 4.2.1 \& 4.2.2. In particular, with $\ell_{h}=\left\vert
S\cap\left(  \left\{  h\right\}  \times S_{s,t}(n,3)\right)  \right\vert $ we
have $\ell_{0}\geq\ell_{1}\geq...\geq\ell_{m-1}$ and $\sum_{h=0}^{m-1}\ell
_{h}=\ell$. We may assume, not only that $S$ is a stable optimal $\ell$-set,
but that $\boldsymbol{\ell}\left(  S\right)  =\left(  \ell_{0},\ell
_{1},...,\ell_{m-1}\right)  $ is Lex-last for such a set. Given $\left\vert
S\right\vert =\ell=\sum_{h=1}^{m}v_{h}m^{n-h}$, the Lex-last $\boldsymbol{\ell
}\left(  S\right)  $ for \textbf{all} $\ell$-sets is
\[
\boldsymbol{\ell}\left(  Lex^{-1}\left(  \left\{  1,2,...,\ell\right\}
\right)  \right)  _{h}=\left\{
\begin{array}
[c]{ll}%
m^{n-1} & \text{if }h<v_{0}\\
\ell-v_{0}m^{n-1} & \text{if }h=v_{0}\\
0 & \text{if }h>v_{0}%
\end{array}
\right.  \text{.}%
\]
If our optimal $S$ is $Lex^{-1}\left(  \left\{  1,2,...,\ell\right\}  \right)
$ we are done. We may assume then that
\[
S\neq Lex^{-1}\left(  \left\{  1,2,...,\ell\right\}  \right)  \text{,}%
\]
so $\boldsymbol{\ell}\left(  S\right)  <$ $\boldsymbol{\ell}\left(
Lex^{-1}\left(  \left\{  1,2,...,\ell\right\}  \right)  \right)  $.

\underline{StOp $\boldsymbol{2}$-Compression}: See Section 4.3. Remember that
the vertices in $V_{I}$ are not actually in $S$ and not counted as such, even
though they can contribute to $\left\vert \Theta_{s,t}(S)\right\vert $.
Similarly for those in $V_{K}$. For each $h$, $0\leq h\leq m-1$, we apply
$Comp_{h}$ to $Comp_{h-1}\left(  Comp_{h-2}\left(  ...\left(  Comp_{0}\left(
S\right)  \right)  \right)  \right)  $. After the application of $Comp_{m-1}$,
every section of $S$ ($S\cap\left(  \left\{  h\right\}  \times S_{s,t}%
(n,3)\right)  $) has been compressed ($=Lex_{\pi}^{-1}\left(  \left\{
1,2,...,\ell_{h}\right\}  \right)  $ where $\pi=I_{h}J_{h}K_{h}$). We then
repeat the cycle of compressions as many times as necessary. By Theorem 4 the
result will eventurally be constant, $Comp_{\infty}\left(  S\right)  $. Also,
$\boldsymbol{\ell}\left(  Comp_{\infty}\left(  S\right)  \right)
=\boldsymbol{\ell}\left(  S\right)  $. However, the compressed set may no
longer be stable. If we re-stabilize it, according to Theorem 6, either
$\boldsymbol{\ell}\left(  Stab_{\infty}(Comp_{\infty}\left(  S\right)
)\right)  >\boldsymbol{\ell}\left(  Comp_{\infty}\left(  S\right)  \right)  $
or $Stab_{\infty}(Comp_{\infty}\left(  S\right)  )$ is compressed. We may
assume then that $Stab_{\infty}\left(  Comp_{\infty}\left(  S\right)  \right)
$ is not only stable but compressed.

\underline{StOp $\boldsymbol{3}$-Subadditivation}: Subadditivation is a
Steiner operation based on the fact that $\left\vert \Theta\right\vert \left(
S(n,3);\ell\right)  $ is subadditive. From StOps $\boldsymbol{1}$ \&
$\boldsymbol{2}$ we may assume that our $\ell$-set $S,$ which minimizes
$\left\vert \Theta\left(  S\right)  \right\vert $ over all $S\subseteq
V_{S_{s,t}(n+1,3)}$ with $\left\vert S\right\vert =\ell$, is stablized,
compressed and has a $3$-tuple $\boldsymbol{\ell}\left(  S\right)  =\left(
\ell_{0},\ell_{1},\ell_{2}\right)  $ that is Lex-last over all such sets. If
$S$ is an initial segment of Lex order we are done. If not, $S$ is still an
ideal of $\mathcal{S}$-$\mathcal{O}\left(  S_{s,t}(n+1,3)\right)  $ and its
intersection, $S\cap\left(  \left\{  h\right\}  \times S_{s,t}(n,3)\right)  $,
is an initial $\ell_{h}$-segment of Lex$_{h}$ order. Let $h_{\min}%
=\min\left\{  h:\ell_{h}<3^{n}\right\}  $ and $h_{\max}=\max\left\{
h:\ell_{h}>0\right\}  $. So since $S\neq Lex^{-1}\left(  \left\{
1,2,...,\ell\right\}  \right)  $, we must have $h_{\min}<h_{\max}$. Then let
$S^{\prime}=S-S\cap\left(  \left\{  h_{\min}\right\}  \times S_{s,t}%
(n,3)\right)  -S\cap\left\{  h_{\max}\right\}  \times S_{s.t}(n,3)$ and we
have
\[
SubAdd(S)=\left\{
\begin{tabular}
[c]{l}%
$S^{\prime}+\left\{  h_{\min}\right\}  \times Lex_{h_{\min}}^{-1}\left(
\ell_{h_{\min}}+\ell_{h_{\max}}\right)  $\\
$\ \ \ \ \ \ \ \ \ \ \ \text{if }\ell_{h_{\min}}+\ell_{h_{\max}}\leq3^{n}$,\\
\\
$S^{\prime}+\left\{  h_{\min}\right\}  \times S_{s,t}(n,3)+\left\{  h_{\max
}\right\}  \times Lex_{h_{\max}}^{-1}\left(  \ell_{h_{\max}}+\ell_{h_{\min}%
}-3^{n}\right)  $\\
$\ \ \ \ \ \ \ \ \ \ \text{ if }\ell_{h_{\min}}+\ell_{h_{\max}}>3^{n}$.
\end{tabular}
\ \right.
\]
In either case $\left\vert SubAdd(S)\right\vert =\left\vert S\right\vert
=\ell$, so $SubAdd$ has property 1 of a StOp. Contributions to the difference,
$\Delta=\left\vert \Theta(SubAdd\left(  S\right)  ))\right\vert -\left\vert
\Theta\left(  S\right)  \right\vert $, come from 3 sources:

\begin{description}
\item[$\Delta_{I}$] The difference in the number of interior edges (within
some $\left\{  h\right\}  \times S(n,3)$) cut by $S$ and $SubAdd\left(
S\right)  )$,

\item[$\Delta_{E}$] The difference in the number of exterior edges (from
$\left\{  h_{1}\right\}  \times S(n,3)$ to $\left\{  h_{2}\right\}  \times
S(n,3)$, $h_{1}\neq h_{2}$) cut by $S$ and $SubAdd\left(  S\right)  )$,

\item[$\Delta_{C}$] The difference in the number of corner edges ($\left\{
v_{h},h^{n+1}\right\}  $) cut by $S$ and $SubAdd\left(  S\right)  $.
\end{description}

And then
\[
\Delta=\Delta_{I}+\Delta_{E}+\Delta_{C}%
\]

\qquad To prove that $SubAdd$ has Property 2 of a StOp ($\left\vert
\Theta\left(  SubAdd\left(  S\right)  \right)  \right\vert \leq\left\vert
\Theta\left(  S\right)  \right\vert $ or equivalently, $\Delta\leq0$), we
reduce to cases: First we consider the possibilities for $\mathfrak{i}%
=S\cap\left(  21^{n}\downarrow\right)  $, where $21^{n}\downarrow$ denotes the
ideal (in $\mathcal{S}$-$\mathcal{O}\left(  S_{s,t}(n+1,3)\right)  $)
generated by (\textit{i.e. }below) $21^{n}$. Since $21^{n}$ is the maximal
element of the component whose minimum element is $01^{n}$, $21^{n}\downarrow$
is exactly that same component. Since $S$ is stable and therefore an ideal in
$\mathcal{S}$-$\mathcal{O}\left(  S_{s,t}(n+1,3)\right)  $, $\mathfrak{i}%
=S\cap\left(  21^{n}\downarrow\right)  $ is an ideal of that component. As we
saw in Example 10, $21^{n}\downarrow$ has $9$ ideals so there are just $9$
possible intersections. Next we consider the possible values of $\left(
s,t\right)  $ with $s,t\geq0$ $\&$ $s+t\leq m$. In general there are
$\binom{m+2}{2}=\left(  m+2\right)  \left(  m+1\right)  /2$ such ordered
pairs. For $m=3$ then there are then 10 pairs. Since the endpoints of all
exterior edges of $S_{s,t}(n+1,3)$ are in $21^{n}\downarrow$, and the other
ends of corner edges are in $I\cup J$, $S\cap\left(  21^{n}\downarrow\right)
$ and $\left(  s,t\right)  $ determine $I_{h},J_{h},K_{h}$ which determine the
relative order of the digits, $0,1$ $\&$ $2$, in the definition of $Lex_{h}$.
Altogether these possibilities give $9\times10=90$ cases to be considered.
However, duality reduces that number by almost half (to $46$, two of the cases
being self-dual) and many of those cases are trivial. Some are not, however
and even have several subcases. For each case a range of values of $\ell
_{0},\ell_{1},\ell_{2}$ are possible. In each case we show that if $S$ is not
already $Lex^{-1}\left(  \left\{  1,2,...,\ell\right\}  \right)  $, then
subadditivation will nontrivially reduce it ($\boldsymbol{\ell}\left(
SubAdd(S)\right)  >_{Lex}\boldsymbol{\ell}(S)$).

For the first four cases below $(1a-1d)$, we proceed step-by-step. After that
we leave out routine steps. In order to verfy the proof, the reader should
fill in the missing steps.

\begin{enumerate}
\item $\mathfrak{i}=\emptyset$

\begin{enumerate}
\item $\left(  s,t\right)  =\left(  0,0\right)  $: Then
\begin{align*}
I_{0}  &  =\emptyset,\text{ }J_{0}=\emptyset\text{ so }K_{0}=\left\{
0,1,2\right\}  \text{ and }0<_{0}1<_{0}2,\\
I_{1}  &  =\emptyset,\text{ }J_{1}=\emptyset\text{ so }K_{1}=\left\{
0,1,2\right\}  \text{ and }0<_{1}1<_{1}2,\\
I_{2}  &  =\emptyset,\text{ }J_{2}=\emptyset\text{ so }K_{2}=\left\{
0,1,2\right\}  \text{ and }0<_{2}1<_{2}2.
\end{align*}

$10^{n}\notin S$ \& $0<_{1}1<_{1}2\Rightarrow\ell_{1}=0\Rightarrow h_{\max}=0$
so the conclusion ($\Delta=0$) is trivial.

\item $\left(  s,t\right)  =\left(  1,0\right)  $: Then
\begin{align*}
I_{0}  &  =\left\{  0\right\}  ,\text{ }J_{0}=\emptyset\text{ so }%
K_{0}=\left\{  1,2\right\}  \text{ and }0<_{0}1<_{0}2\text{,}\\
I_{1}  &  =\emptyset,\text{ }J_{1}=\emptyset\text{ so }K_{1}=\left\{
0,1,2\right\}  \text{ and }0<_{1}1<_{1}2\text{,}\\
I_{2}  &  =\emptyset,\text{ }J_{2}=\emptyset\text{ so }K_{2}=\left\{
0,1,2\right\}  \text{ and }0<_{2}1<_{2}2\text{.}%
\end{align*}

$10^{n}\notin S$ \& $0<_{1}1<_{1}2\Rightarrow\ell_{1}=0\Rightarrow h_{\max}=0$
so the conclusion ($\Delta=0$) is trivial.

\item $\left(  s,t\right)  =\left(  2,0\right)  $: Then
\begin{align*}
I_{0}  &  =\left\{  0\right\}  ,\text{ }J_{0}=\emptyset\text{ so }%
K_{0}=\left\{  1,2\right\}  \text{ and }0<_{0}1<_{0}2\text{,}\\
I_{1}  &  =\left\{  1\right\}  ,\text{ }J_{1}=\emptyset\text{ so }%
K_{1}=\left\{  0,2\right\}  \text{ and }1<_{1}0<_{1}2\text{,}\\
I_{2}  &  =\emptyset,\text{ }J_{2}=\emptyset\text{ so }K_{2}=\left\{
0,1,2\right\}  \text{ and }0<_{2}1<_{2}2\text{.}%
\end{align*}
$20^{n}\notin S$ \& $0<_{2}1<_{2}2\Rightarrow\ell_{2}=0$, so $h_{\max}=1$.
$01^{n}\notin S$ \& $0<_{0}1<_{0}2\Rightarrow\ell_{0}<3^{n}/2$.

\ \ \ \ $(i)$ If $\ell_{0}+\ell_{1}<3^{n}/2$ then $\Delta_{E}=0$ \&
$\Delta_{C}=1$ so $\Delta\leq0$ by Theorem

\ \ \ \ \ \ \ \ \ 7.

\ \ \ \ $(ii)$ If $\ell_{0}+\ell_{1}>3^{n}/2$ then $\Delta_{E}=1$ \&
$\Delta_{C}=1$ but $\Delta\leq0$ by Lemma

\ \ \ \ \ \ \ \ \ 10.

\item $\left(  s,t\right)  =\left(  3,0\right)  $: Then
\begin{align*}
I_{0}  &  =\left\{  0\right\}  ,\text{ }J_{0}=\emptyset\text{ so }%
K_{0}=\left\{  1,2\right\}  \text{ and }0<_{0}1<_{0}2\text{,}\\
I_{1}  &  =\left\{  1\right\}  ,\text{ }J_{1}=\emptyset\text{ so }%
K_{1}=\left\{  0,2\right\}  \text{ and }1<_{1}0<_{1}2\text{,}\\
I_{2}  &  =\left\{  2\right\}  ,\text{ }J_{2}=\emptyset\text{ so }%
K_{2}=\left\{  0,2\right\}  \text{ and }2<_{2}0<_{2}1\text{.}%
\end{align*}
$01^{n}\notin S$ \& $0<_{0}1<_{0}2\Rightarrow\ell_{0}<3^{n}/2$. Also,
$\ell_{2}\leq\ell_{1}\leq\ell_{0}$.

\ \ \ \ $(i)$ If $\ell_{2}=0$, the result ($\Delta=0$) follows as in Case $1c$.

\ \ \ \ $(ii)$ If $\ell_{2}>0$, $h_{\min}=0$ \& $h_{\max}=2$.

\ \ \ \ \ \ \ \ \ $\left(  A\right)  $ If $\ell_{0}+\ell_{2}<3^{n}/2$ then
$\Delta_{E}=0$ \& $\Delta_{C}=1$ so $\Delta\leq0$ by

\ \ \ \ \ \ \ \ \ \ \ \ \ Theorem 7.

\ \ \ \ \ \ \ \ \ $\left(  B\right)  $ If $\ell_{0}+\ell_{2}>3^{n}/2$ then
$\Delta_{E}=1$ \& $\Delta_{C}=1$ but $\Delta\leq0$ by

\ \ \ \ \ \ \ \ \ \ \ \ \ \ Lemma 10.

\item $\left(  s,t\right)  =\left(  0,1\right)  $: Then $0<_{h}1<_{h}2$, for
$h=0,1,2$ and the result ($\Delta=0$) follows as in Case 1a.

\item $\left(  s,t\right)  =\left(  1,1\right)  $: Then $0<_{h}1<_{h}2$, for
$h=0,2$ but $1<_{1}0<_{1}2$. $0<\ell_{1}\leq\ell_{0}<3^{n}/2$ \& $\ell_{2}=0$.
$h_{\min}=0$ $\&$ $h_{\max}=1$. $\Delta_{E}\leq1$ \& $\Delta_{C}=0$, so by
Theorem 7, $\Delta\leq0$.

\item $\left(  s,t\right)  =\left(  2,1\right)  $: Then $0<_{0}1<_{0}2$,
$1<_{1}0<_{1}2$ and $2<_{2}0<_{2}1$. The result ($\Delta\leq0$) follows as in
Case 1d.

\item $\left(  s,t\right)  =\left(  0,2\right)  $: Then $0<_{h}1<_{h}2$, for
$h=0,2$ but $1<_{1}0<_{1}2$. The result ($\Delta\leq0$) follows as in Case 1f.

\item $\left(  s,t\right)  =\left(  1,2\right)  $: Then $0<_{0}1<_{0}2$,
$1<_{1}0<_{1}2$ and $2<_{2}0<_{2}1$. The result ($\Delta\leq0$) follows as in
Case 1d.

\item $\left(  s,t\right)  =\left(  0,3\right)  $: Then $0<_{0}1<_{0}2$,
$1<_{1}0<_{1}2$ and $2<_{2}0<_{2}1$. The result ($\Delta\leq0$) follows as in
Case 1d.
\end{enumerate}

\item $\mathfrak{i}=\left\{  01^{n}\right\}  $: Since $01^{n}\in S$,
$10^{n}\in I_{1}$ and so $Lex_{1}^{-1}\left(  1\right)  =10^{n}$ no matter
what $\left(  s,t\right)  $ is. Since $10^{n}\notin S$, $\ell_{1}=0=\ell_{2}$
and our result is trivial ($\Delta=0$).

\item $\mathfrak{i}=\left\{  01^{n},02^{n}\right\}  $: The same reasoning as
for Case 2 applies and $\Delta=0$.

\item $\mathfrak{i}=\left\{  01^{n},10^{n}\right\}  $:

\begin{enumerate}
\item $\left(  s,t\right)  =\left(  0,0\right)  $: Then $1<_{0}0<_{0}2$ \& for
$h=1,2$, $0<_{h}1<_{h}2$. Therefore $0<\ell_{1}\leq\ell_{0}<3^{n}$ \&
$\ell_{2}=0$. Therefore $h_{\min}=0$ \& $h_{\max}=1$.

$\quad(i)$ If $\ell_{0}<3^{n}/2$ $\left(  \Rightarrow\ell_{1}<3^{n}/2\text{ \&
}\ell_{0}+\ell_{1}<3^{n}\right)  $,

\ \ \ \ \ \ \ \ \ $(A)$ If $\ell_{0}+\ell_{1}<3^{n}/2$, $\Delta_{E}=1$ \&
$\Delta_{C}=0$ so Theorem 7 implies

\ \ \ \ \ \ \ \ \ \ \ \ \ that $\Delta\leq0$.

\ \ \ \ \ \ \ \ $\ \left(  B\right)  $ If $\ell_{0}+\ell_{1}>3^{n}/2$,
$\Delta_{E}=1$ \& $\Delta_{C}=1$ but Lemma 10 implies

\ \ \ \ \ \ \ \ \ \ \ \ \ that $\Delta\leq0$

$\quad(ii)$ If $\ell_{0}>3^{n}/2$,

$\ \ \ \ \ \ \ \ (A)$ If $\ell_{0}+\ell_{1}<3^{n}$ ($\Rightarrow\ell_{1}%
<3^{n}/2$), $\Delta_{E}=1$ \& $\Delta_{C}=0$ so

\ \ \ \ \ \ \ \ \ \ \ \ \ Theorem 7 implies that $\Delta\leq0$.

\ \ \ \ \ \ \ $(B)$ If $\ell_{0}+\ell_{1}=3^{n}$, $\Delta_{E}=2$ \&
$\Delta_{C}=0$, but by Theorem 2,

$\ \ \ \ \ \ \ \ \ \ \ \ \Delta_{I}\leq0-\left(  2+2\right)  =-4$, so
\begin{align*}
\Delta &  =\Delta_{I}+\Delta_{E}+\Delta_{C}\\
&  \leq\left(  -4\right)  +2+0\\
&  =-2<0\text{. }%
\end{align*}

\ \ \ \ \ \ \ \ \ $(C)$ If $\ell_{0}+\ell_{1}>3^{n}$, then $\Delta_{E}=1$ \&
$\Delta_{C}\leq0$, so Theorem 7

\ \ \ \ \ \ \ \ \ \ \ \ \ implies that $\Delta\leq0$.

\item $\left(  s,t\right)  =\left(  1,0\right)  $: Then $0<_{h}1<_{h}2$, for
$h=0,1,2$. Therefore $3^{n}/2<\ell_{0}<3^{n}$, $0<\ell_{1}\leq$ $\ell_{0}$ \&
$\ell_{2}=0$, so $h_{\min}=0$ \& $h_{\max}=1$.

$\quad(i)$ If $\ell_{1}<3^{n}/2$

\ \ \ \ \ \ \ \ \ $(A)$ If $\ell_{0}+\ell_{1}<3^{n}$, $\Delta_{E}=1$ \&
$\Delta_{C}=0$ so Theorem 7 implies

\ \ \ \ \ \ \ \ \ \ \ \ \ that $\Delta\leq0$.

\ \ \ \ \ \ \ \ \ $(B)$ If $\ell_{0}+\ell_{1}=3^{n}$, $\Delta_{E}=2$ \&
$\Delta_{C}=0$ but by Theorem 2,

$\ \ \ \ \ \ \ \ \ \ \ \ \ \Delta\leq0$.

\ \ \ \ \ \ \ \ \ $(C)$ If $\ell_{0}+\ell_{1}>3^{n}$, $\Delta_{E}=1$ \&
$\Delta_{C}\leq0$ so Theorem 7 implies

\ \ \ \ \ \ \ \ \ \ \ \ \ that $\Delta\leq0$.

$\quad(ii)$ If $\ell_{1}>3^{n}/2$ ($\Rightarrow\ell_{0}+\ell_{1}>3^{n}$), then
$\Delta_{E}=1$ \& $\Delta_{C}\leq0$ so by

\ \ \ \ \ \ \ Theorem 7, $\Delta\leq0$.

\item $\left(  s,t\right)  =\left(  2,0\right)  $: Then $0<_{h}1<_{h}2$, for
$h=0,1,2$. Therefore $3^{n}/2<\ell_{0}<3^{n}$, $0<\ell_{1}\leq$ $\ell_{0}$ \&
$\ell_{2}=0$, so $h_{\min}=0$ \& $h_{\max}=1$.

$\quad(i)$ If $\ell_{1}<3^{n}/2$ the result follows essentially as for
$\left(  4bi\right)  $.

$\quad(ii)$ If $\ell_{1}>3^{n}/2$ ($\Rightarrow\ell_{0}+\ell_{1}>3^{n}$),

\ \ \ \ \ \ \ \ \ $(A)$ If $\ell_{0}+\ell_{1}<3^{n+1}/2$, $\Delta_{E}=1$ \&
$\Delta_{C}=1$ but by the dual of

\ \ \ \ \ \ \ \ \ \ \ \ \ Lemma 10 (Corollary 5), $\Delta\leq0$.

\ \ \ \ \ \ \ \ \ $(B)$ If $\ell_{0}+\ell_{1}>3^{n+1}/2$, $\Delta_{E}=1$ \&
$\Delta_{C}=0$ so by Theorem 7,

$\ \ \ \ \ \ \ \ \ \ \ \ \ \Delta\leq0$.

\item $\left(  s,t\right)  =\left(  3,0\right)  $: Then $0<_{h}1<_{h}2$, for
$h=0,1$ \& $2<_{2}0<_{2}1$. $3^{n}/2<\ell_{0}<3^{n}$, $0<\ell_{1}\leq$
$\ell_{0}$ \& $\ell_{2}\leq\min\left\{  \ell_{1},3^{n}/2\right\}  $.

$\quad(i)$ If $\ell_{2}=0$ then the result follows as for $\left(  4c\right)
$.

$\quad(ii)$ If $\ell_{2}>0$ then $h_{\min}=0$ \& $h_{\max}=2$,

$\qquad\quad(A)$ If $\ell_{0}+\ell_{2}<3^{n}$, then $\Delta_{E}=0$ $\&$
$\Delta_{C}=1$ so by Theorem 7,

$\ \ \ \ \ \ \ \ \ \ \ \ \ \Delta\leq0$.

\ \ \ \ \ \ \ \ \ $(B)\ \ $If $\ell_{0}+\ell_{2}=3^{n}$, then $\Delta_{E}=1$
$\&$ $\Delta_{C}=1$ so by Theorem 2,

$\ \ \ \ \ \ \ \ \ \ \ \ \ \Delta\leq0$.

\ \ \ \ \ \ \ \ \ $\left(  C\right)  $ \ If $\ell_{0}+\ell_{2}>3^{n}$, then
$\Delta_{E}=1$ $\&$ $\Delta_{C}=0$ so by Theorem 7,

$\ \ \ \ \ \ \ \ \ \ \ \ \ \Delta\leq0$.

\item $\left(  s,t\right)  =\left(  0,1\right)  $: Then $1<_{0}0<_{0}2$ \& for
$h=0,1$, $0<_{h}1<_{h}2$. Therefore $0<\ell_{0}<3^{n}$, $0<\ell_{1}\leq
\ell_{0}$ \& $\ell_{2}=0$, so $h_{\min}=0$ \& $h_{\max}=1$.

$\quad(i)$ If $\ell_{0}+\ell_{1}<3^{n}$ then $\Delta_{E}=1$ $\&$ $\Delta
_{C}=0$ so Theorem 7 implies

\ \ \ \ \ \ \ that $\Delta\leq0$.

$\quad(ii)$ If $\ell_{0}+\ell_{1}=3^{n}$, then $\Delta_{E}=2$ $\&$ $\Delta
_{C}=0$ but by Theorem 2,

$\ \ \ \ \ \ \ \Delta\leq0$.

$\quad(iii)$ If $\ell_{0}+\ell_{1}>3^{n}$, then $\Delta_{E}=1$ $\&$
$\Delta_{C}=0$ so Theorem 7 implies

\ \ \ \ \ \ \ that $\Delta\leq0$.

\item $\left(  s,t\right)  =\left(  1,1\right)  $: Then $0<_{h}1<_{h}2$, for
$h=0,1,2$. Therefore $3^{n+1}/2<\ell_{0}<3^{n+1}$, $0<\ell_{1}\leq$ $\ell_{0}$
\& $\ell_{2}=0$. It then follows as in Case 4c that $\Delta\leq0$.

\item $\left(  s,t\right)  =\left(  2,1\right)  $: Then for $h=0,1$,
$0<_{h}1<_{h}2$ \& $2<_{2}0<_{2}1$. Therefore, $3^{n}/2<\ell_{0}<3^{n}$,
$0<\ell_{1}\leq$ $\ell_{0}$ \& $\ell_{2}\leq\min\left\{  \ell_{1}%
,3^{n}/2\right\}  $. Therefore, $h_{\min}=0$.

$\quad(i)$ If $\ell_{2}=0$ then the result, $\Delta\leq0$, follows as in 4c.

$\quad(ii)$ If $\ell_{2}>0$ then $h_{\max}=2$. Whatever $\ell_{0}+\ell_{2}$
is, $\Delta_{E}\leq1$ \& $\Delta_{C}=0$

\ \ \ \ \ \ \ so Theorem 7 implies that $\Delta\leq0$.

\item $\left(  s,t\right)  =\left(  0,2\right)  $: Then $1<_{0}0<_{0}2$ and
for $h=0,1$, $0<_{h}1<_{h}2$. Therefore $0<\ell_{0}<3^{n}$, $0<\ell_{1}%
\leq\ell_{0}$ \& $\ell_{2}=0$, so $h_{\min}=0$ \& $h_{\max}=1$. The subsequent
cases are the same as Case 4e.

\item $\left(  s,t\right)  =\left(  1,2\right)  $: Then for $h=0,1$,
$0<_{h}1<_{h}2$ \& $2<_{2}0<_{2}1$. Therefore, $3^{n}/2<\ell_{0}<3^{n}$,
$0<\ell_{1}\leq$ $\ell_{0}$ \& $\ell_{2}\leq\min\left\{  \ell_{1}%
,3^{n}/2\right\}  $. Therefore, $h_{\min}=0$.

$\quad(i)$ If $\ell_{2}=0$ ($\Rightarrow h_{\max}=1$), the result, $\Delta
\leq0$, follows as in 4c.

$\quad(ii)$ If $\ell_{2}>0$, ($\Rightarrow h_{\max}=2$), then $\Delta_{E}%
\leq1$ $\&$ $\Delta_{C}=0$ so Theorem 7

\ \ \ \ \ \ \ implies that $\Delta\leq0$.

\item $\left(  s,t\right)  =\left(  0,3\right)  $: Then $1<_{0}0<_{0}2$,
$0<_{1}1<_{1}2$ \& $2<_{2}0<_{2}1$. Therefore $3^{n}/2<\ell_{0}<3^{n}$,
$0<\ell_{1}\leq$ $\ell_{0}$ \& $\ell_{2}\leq\min\left\{  \ell_{1}%
,3^{n}/2\right\}  $ so $h_{\min}=0$. The result, $\Delta\leq0$, then follows
as in Case 4i.
\end{enumerate}

\item $\mathfrak{i}=\left\{  01^{n},10^{n},02^{n}\right\}  $: All these cases
are trivial ($\Delta=0$) because
\[
\left(  0\in I_{2}\text{ \& }20^{n}\notin S\right)  \Rightarrow\left(
\ell_{2}=0\right)
\]
and
\[
(2\in K_{0}\text{ \& }02^{n}\in S)\Rightarrow\left(  \ell_{0}=3^{n}\right)  .
\]

\end{enumerate}

All other stable \& compressed sets are dual to one of those we have
considered above, so our proof is complete.
\end{proof}
\end{theorem}

\section{Conclusions \& Comments}

\subsection{Possibilities for $m>3$}

The complexity of the final phase of the arguement (Section 5) prohibited us
from carrying it through for $m>3.$ For $m=3$ the initial estimate of the
number of cases was $90$ (the actual number of cases (last part of Section 5)
was 41). The corresponding estimate for $m=4$ is the product of $28$
($=\left\vert \mathfrak{I}\left(  \mathcal{S}\text{-}\mathcal{O}\left(
32\downarrow\right)  \right)  \right\vert $, see Figure 5) by $15$
($=\binom{4+2}{2}$) which is $420$. $m=3$ already pushed our limits for hand
computation so we are hoping to do $m\geq4$ by computer.

\subsection{Our Logical Strategy (3 \textit{StOps})}

\begin{enumerate}
\item This strategy is modeled on that of the first proof of Theorem 1.1 in
\cite{Har04}. That proof for had its origins in our first paper (1962) but was
only included in the monograph \cite{Har04} to show how complicated proofs of
such theorems could be. It was sufficiently complicated that in the original
paper we missed a case. Fortunately A. J. Bernstein noticed the oversight and
filled in the missing arguement (Lemma 1.3 of \cite{Har04}). After developing
the theory of Steiner operations in Chapters 2 \& 3 of \cite{Har04},
stabilization and compression were used to give a relatively short (and easily
verified) proof of Theorem 1.1 (Sec. 3.3.5 of \cite{Har04}). In this paper
Bernstein's lemma developed into the Steiner operation subaddification.

\qquad More recently, in \cite{Har15}, we showed that initial $\ell$-segments
of Hales order (see \cite{Har04}, p. 56) maximize $Type$ over all stable
$\ell$-sets of vertices of the $n$-cube, $Q_{n}$. That proof, based on the
self-similarity of the stabilization order of $Q_{n}$ (also known as its
Bruhat order (see \cite{Har04}, Sect. 5.2)) seemed novel at first, until we
realized that the first proof of Theorem 1.1 in \cite{Har04}, by not using the
full power inherent in compression, treated $Q_{n}$ as a self-similar
structure. This suggested the possibility of proving isoperimetric theorems
for other self-similar structures. The first target for our project to solve
isoperimetric problems on other self-similar structures was the $EIP$ on the
Sierpinski gasket graph. $SG_{n}$ is self-similar but has no product
decomposition and relatively little symmetry (so a proof technique beyond
compression and stabilization is required).

\item The proof of our Main Theorem is by induction, but the reduction of
$S_{s,t}(n+1,3)$ to $S_{s,t}(n,3)$ is based on three different Steiner
operations, stabilization, compression \& subadditivation. Each reduces the
number of possible solution $\ell$-sets, the last reducing it to just one,
$Lex^{-1}\left(  \left\{  1,2,...,\ell\right\}  \right)  $. This logic, which
follows naturally from the notion of morphism, is essentially the same as that
of Steiner's "proof" of the classical isoperimetric theorem in the Euclidean
plane (c. 1840): If $S$ is any closed set other than a disk, it can be
transformed by symmetrization (the original Steiner operation) to a set having
the same area and smaller boundary. Therefore the only possible solution set
(up to isomorphism) is a disk.

Steiner's symmetrization was (as far as we know) the first noninvertable
morphism in mathematics. Symmetry and similarity were known to Euclid, of
course, and Galois made use of symmetry. However, as Weierstrass pointed out,
Steiner's epoch-making insight was incomplete: It was still logically possible
that the isoperimetric problem had no solution; that the greatest lower bound
of all boundary lengths might not be achieved by any set. Weierstrass was a
pioneer in functional analysis but it took another 40 years to clear up this
last detail. Combinatorial StOps do not suffer from this problem. By
finiteness, if there is only one set fixed by a StOp, it has to be a solution.

\item The flexability and adaptability of Steiner operations continues to be
amazing \& gratifying. The 3 StOps in the proof of our theorem, though based
on those in the proof of Theorem 1.1 of \cite{Har04}, had to be substantially
modified to accomplish the purpose. They held up well. However, the additional
complications made the complexity of its role model, the proof of Theorem 1.1
of \cite{Har04}, seem insignifcant by comparison.
\end{enumerate}

\subsection{A Coincidence?}

$K_{m}^{n}$ and $S(n,m)$ have the same set of vertices, $\left\{
0,...,m-1\right\}  ^{n}$ but very different sets of edges. It is curious that
their $EIPs$ have a common solution, initial segments of Lex order.

\subsection{The Nested Solutions Property}

The nested solutions property is shared by most isoperimetric problems, finite
and continuous, that have been solved. However, continuous isoperimetric
problems may also be solved by variational means, even if they do not have
nested solutions. Variational methods do not work well for combinatorial
isoperimetric problems without nested solutions because the solution sets
become chaotic near the break. What has worked on several combinatorial
isoperimetric problems lacking nested solutions is passage to a continuous
limit and solving the resulting continuous isoperimetric problem (see Chapter
10 of \cite{Har04}). With Sierpinski graphs however, we have (surprisingly)
the opposite situation: $S(n,m)$, the generalized \& expanded Sierpinski
graph, is fractal but solutions of its isoperimetric problems are nested. The
main challenge was to adapt compression to the recursive structure of $S(n,m)$.

\subsection{The Isoperimetric Problem on SG$_{\infty}$}

The history of continuous variational problems (such as the classical
isoperimetric problem in the plane and the brachystochrone problem) shows that
they present two challenging questions:

\begin{enumerate}
\item What is the solution?

\item (Assuming we "know" the solution) Can we prove it?
\end{enumerate}

Galileo erred on question 1, guessing that the solution of the brachystochrone
problem was the arc of a circle. And Archimedes knew that the solution of the
isoperimetric problem was the disk but a logically rigorous proof eluded
mathematicians until the late 19$^{th}$ century. So what are the answers to
these questions for the isoperimetric problem on the Sierpinski gasket?

For Question 1 we claim that $SG_{\infty}$ has nested solutions given by the
function, $\eta^{-1}:\left[  0,1\right]  \rightarrow SG_{\infty}$, defined by
\[
\eta^{-1}\left(  a\right)  =\left(  \sum_{a_{i}=0}2^{-i},\sum_{a_{i}=1}%
2^{-i},\sum_{a_{i}=2}2^{-i}\right)  \text{,}%
\]
where $a=\sum_{i=1}^{\infty}a_{i}3^{-i}$ is the base $3$ representation of
$a\in\left[  0,1\right]  $ and $SG_{\infty}$ is constructed in $\mathbb{R}^{3}
$ starting with the triangle whose vertices are $\left(  1,0,0\right)  $,
$\left(  0,1,0\right)  $ \& $\left(  0,0,1\right)  $. If $a$ is a triadic
rational then it has two base $3$ representations (such as $1/3=\sum_{i=1}%
^{1}3^{-i}=\sum_{i=2}^{\infty}2\cdot3^{-i}$). In that case use the infinite
one in calculating $\eta^{-1}\left(  a\right)  $.

\begin{example}
$\eta^{-1}\left(  1/3\right)  $%
\begin{align*}
&  =\eta^{-1}\left(  \sum_{i=2}^{\infty}2\cdot3^{-i}\right) \\
&  =\left(  \sum_{i=1}^{1}2^{-i},0,\sum_{i=2}^{\infty}2^{-i}\right) \\
&  =\left(  1/2,0,1/2\right)  \text{.}%
\end{align*}
$.$
\end{example}

\begin{example}
$\eta^{-1}\left(  1/2\right)  $%
\begin{align*}
&  =\eta^{-1}\left(  \sum_{i=1}^{\infty}1\cdot3^{-i}\right) \\
&  =(0,\sum_{i=1}^{\infty}2^{-i},0)\\
&  =\left(  0,1,0\right)  .
\end{align*}

\end{example}

\begin{example}
$\eta^{-1}\left(  1/6\right)  $%
\begin{align*}
&  =\eta^{-1}\left(  \sum_{i=2}^{\infty}1\cdot3^{-i}\right) \\
&  =(\sum_{i=1}^{1}2^{-i},\sum_{i=2}^{\infty}2^{-i},0)\\
&  =\left(  1/2,1/2,0\right)  .
\end{align*}

\end{example}

\bigskip This function, $\eta^{-1}$, is the limit, as $n\rightarrow\infty$, of
the composition of $Lex^{-1}:\left\{  1,2,...,3^{n}\right\}  $ $\rightarrow
S(n,3)$ with the embedding, $y:S(n,3)\rightarrow\mathbb{R}^{3}$, of Section
3.1.3. The insight behind our claim is that as $n\rightarrow\infty$, the $EIP$
on $S(n,3)$ converges to the natural isoperimetric problem on the Sierpinski
gasket. Intuitively, as $n\rightarrow\infty$, and $S(n,3)\rightarrow
SG_{\infty}$, the edges, of length $1/2^{n}$, are shrinking to $0$. So that,
in the limit, the edge-boundary becomes the topological boundary. The
isoperimetric profile of the Sierpinski gasket is then the limit of the
edge-isoperimetric profile of $S(n,3)$ as $n\rightarrow\infty$: If
$\lambda\left(  a\right)  $ denotes the minimum "length" of the topological
boundary of any closed set in $SG_{\infty}$ of "area" $a$, then
\[
\lambda\left(  a\right)  =\left\{
\begin{array}
[c]{ll}%
\left\vert \Theta\right\vert \left(  S\left(  n,3\right)  ;\ell\right)  &
\text{if }a=\ell/3^{n}\text{, a triadic rational,}\\
\omega\text{ } & \text{(countable }\infty\text{) otherwise.}%
\end{array}
\right.
\]
For a triadic rational, $\ell/3^{n}$, $0<\ell<3^{n}$, $\lambda\left(
a\right)  $ is actually given by the formula of Proposition 1 with $m=3:$%
\[
\lambda\left(  \ell/3^{n}\right)  =\sum_{h=1}^{n}\ell_{h}\left(  3-\ell
_{h}\right)  +\left\vert \ell_{h}^{\prime}-\ell_{h}\right\vert -\ell
_{h}\text{,}%
\]
where $0<\ell<3^{n}$,
\begin{align*}
\ell+1  &  =\sum_{h=1}^{n}\ell_{h}3^{n-h}\text{ and}\\
\ell_{h}^{\prime}  &  =1+\max\left\{  j:j^{n-h}\leq_{Lex}\left(  \ell
_{h+1},\ell_{h+2},...,\ell_{n}\right)  \right\}  .
\end{align*}
Note that $\lambda\left(  a\right)  $ is countable infinity except at a
countably infinite set (which is necessarily of measure zero!)!

This result, the solution of a continuous isoperimetric problem by
combinatorial means, is the realization of a longheld fantasy of the author.
Having appropriated so much from the classical analytic theory of
isoperimetric problems, it is gratifying to be able to give something back.
Also, we are grateful to Michel Lapidus, UCR colleague and expert on fractal
geometry, for his encouragement and intellectual support of this project.

\section{\bigskip Appendix}

\subsection{How Many Components in $\mathcal{S}$-$\mathcal{O}\left(
S(n,m)\right)  ?$}

A vertex $v\in\left\{  0,1,...,m-1\right\}  ^{n}$ may be thought of as an
ordered \ partition of $\left\{  1,...,n\right\}  $ into $m$ blocks:
$v=\left(  v_{1},v_{2},...,v_{n}\right)  $ corresponds to $p\left(  v\right)
=\left(  v^{-1}\left(  0\right)  ,v^{-1}\left(  1\right)  ,...,v^{-1}\left(
m-1\right)  \right)  $. Note that if $i$ does not appear in $v$, then
$v^{-1}\left(  i\right)  =\emptyset$, the empty set. The covering relations,
$v\lessdot v^{\prime}$, of $\mathcal{S}$-$\mathcal{O}\left(  S(n,m)\right)  $
are given by transpositions of consecutive integers, $i\left(  i+1\right)  $
such that $\min\left\{  j:v_{j}=i\right\}  <\min\left\{  j:v_{j}=i+1\right\}
$ operating on the coordinates of $v$ to give $v^{\prime}$. Those
transpositions generate all permutations of the components of $p\left(
v\right)  $ so components are characterized by their common unordered
\ partition of $\left\{  1,...,n\right\}  $ into $m$ or fewer blocks. Every
component has a unique minimum element in $C_{0}^{\prime}$, so the number of
components is the same as $\left\vert C_{0}^{\prime}\right\vert $. This number
is $\sum_{k=1}^{m}S_{n,k}$, $S_{n,k}$ being the Sterling number of the second
kind. The $S_{n,k}$'s are well known (see The On-line Encyclopedia of Integer Sequences).

\subsection{The Derivation of Compression for Products}

In \cite{Har04} compression is a Steiner operation on a product of graphs,
$G\times H$, based on at least one of the factors (say $G$) having nested
solutions. If $\left\vert S\cap\left(  G\times\left\{  w\right\}  \right)
\right\vert =\ell_{w}$, then

\begin{description}
\item[Question] \bigskip What lower bound can be inferred on $\left\vert
\Theta\left(  S\right)  \right\vert $ for $S\subseteq V_{G\times H}$,
$\left\vert S\right\vert =\ell^{\prime}$?

\item[Answer] $\left\vert \Theta\left(  S\right)  \right\vert \geq\sum_{w\in
V_{H}}\left\vert \Theta\right\vert \left(  G;\ell_{w}\right)  +\sum_{\left\{
w_{1},w_{2}\right\}  \in E_{H}}\left\vert \ell_{w_{1}}-\ell_{w_{2}}\right\vert
$.
\end{description}

This lower bound can be achieved if $\eta:V_{G}\rightarrow\left\{
1,2,...,\left\vert V_{G}\right\vert \right\}  $ is a numbering of the vertices
of $G$ such that $\left\vert \Theta\left(  \eta^{-1}\left(  \left\{
1,2,...,\ell\right\}  \right)  \right)  \right\vert =\left\vert \Theta
\right\vert \left(  G;\ell\right)  $. The existence of such a numbering is the
definition of nested solutions. See Section 1.2.2.

We define \textbf{compression} by
\[
Comp_{\eta,G\times H}\left(  S\right)  =%
{\displaystyle\bigcup\limits_{w\in V_{H}}}
\left(  \eta^{-1}\left(  \left\{  1,2,...,\ell_{w}\right\}  \right)
\times\left\{  w\right\}  \right)
\]
and then

\begin{theorem}
(Theorem 3.4 of \cite{Har04}) $Comp_{\eta,G\times H}$ is a Steiner operation.
\end{theorem}

Property 1 is trivial but the crucial Property 2 (of a StOp) is also easy to
verify: $\eta^{-1}\left(  \left\{  1,2,...,\ell_{w}\right\}  \right)
\times\left\{  w\right\}  $ minimizes edges cut within $G\times\left\{
w\right\}  $ by a set of cardinality $\ell_{w}$ and all edges between
$G\times\left\{  w_{1}\right\}  $ and $G\times\left\{  w_{2}\right\}  $ are of
the form $\left\{  \left(  v,w_{1}\right)  ,\left(  v,w_{2}\right)  \right\}
$. The obvious lower bound on the number of such edges cut by any $S$ given
$\ell_{w_{1}},\ell_{w_{2}}$ is$\ \left\vert \ell_{w_{1}}-\ell_{w_{2}%
}\right\vert $, which is achieved by $Comp_{\eta,G\times H}\left(  S\right)  $.

\bigskip

\bigskip
\end{document}